\documentclass[11pt,reqno]{amsart}
\usepackage{amssymb,mathrsfs,graphicx}
\usepackage{ifthen}
\usepackage{colortbl}
\definecolor{black}{rgb}{0.0, 0.0, 0.0}
\definecolor{red}{rgb}{1.0, 0.5, 0.5}
\provideboolean{shownotes} % define a boolean variable
\setboolean{shownotes}{true} % defined as ``true'' or ``false''
\newcommand{\margnote}[1]{
\ifthenelse{\boolean{shownotes}}%
{\marginpar{\raggedright\tiny\texttt{#1}}}%
{}%
}
\newcommand{\hole}[1]{
\ifthenelse{\boolean{shownotes}}%
{\begin{center} \fbox{ \rule {.25cm}{0cm} \rule[-.1cm]{0cm}{.4cm}
\parbox{.85\textwidth}{\begin{center} \texttt{#1}\end{center}} \rule
{.25cm}{0cm}}\end{center}} {} }
\title[Finite-time blow-up for Vlasov/Navier-Stokes equations and related systems]{Finite-time blow-up phenomena of Vlasov/Navier-Stokes equations and related systems}

\author[Choi]{Young-Pil Choi}
\address[Young-Pil Choi]{\newline Fakult\"at Mathematik, Technische Universit\"at M\"unchen \newline
Boltzmannstra{\ss}e 3, 85748, Garching bei M\"unchen, Germany}
\email{ychoi@ma.tum.de}

\numberwithin{equation}{section}

\newtheorem{theorem}{Theorem}[section]
\newtheorem{lemma}{Lemma}[section]
\newtheorem{corollary}{Corollary}[section]

\newtheorem{remark}{Remark}[section]
\newtheorem{definition}{Definition}[section]

\newcommand{\R}{\mathbb R}
\newcommand{\ms}{\mathbb S}

\newcommand{\T}{\mathbb T}

\newcommand{\bq}{\begin{equation}}
\newcommand{\eq}{\end{equation}}

\newcommand{\lt}{\left}
\newcommand{\rt}{\right}

\newcommand{\pa}{\partial}

\newcommand{\ml}{\mathcal{L}}

\newcommand{\into}{\int_{\R^d}}

\newcommand{\intr}{\int_{\R^d}}
\newcommand{\intor}{\int_{{\R^d} \times \R^d}}

\newcommand{\intrr}{\int_{\R^d \times \R^d}}

\def\charf {\mbox{{\text 1}\kern-.30em {\text l}}}
\begin{document}
%%%%%%%%%%%%%%%%
\allowdisplaybreaks

\date{\today}

\subjclass[]{}
\keywords{Sprays, blow-up, kinetic-fluid, two-phase fluid, multiphase flows, Cauchy problems}

%\thanks{\textbf{Acknowledgments.}  }

\begin{abstract} This paper deals with the finite-time blow-up phenomena of classical solutions for Vlasov/Navier-Stokes equations under suitable assumptions on the initial configurations. We show that a solution to the coupled kinetic-fluid system may be initially smooth, however, it can become singular in a finite period of time. We provide a simple idea of showing the finite time blow up of classical solutions to the coupled system which has not been studied so far. We also obtain analogous results for related systems, such as isentropic compressible Navier-Stokes equations, two-phase fluid equations consisting of pressureless Euler equations and Navier-Stokes equations, and thick sprays model.
\end{abstract}

\maketitle \centerline{\date}

\tableofcontents

%%%%%%%%%%%%%%%%%%%%%%%%%%%%%%%%%%%%%%%%%%%%%%%%%%%%%%%%%%%%%%%%%%%%%%%%%%%%%%%%%5
%
%
%                        Section: Introduction
%
%
%%%%%%%%%%%%%%%%%%%%%%%%%%%%%%%%%%%%%%%%%%%%%%%%%%%%%%%%%%%%%%%%%%%%%%%%%%%%%%%%%
\section{Introduction}
Sprays are complex flows which are constituted of dispersed particles such as droplets, dust, etc, in an underlying gas. A coupling of particles and gas was first proposed by Williams \cite{Will}. Later O'Rouke \cite{Rou} classified the sprays depending on the volume fraction of the gas; thin sprays in which the volume occupied by the particles is negligible compared to the volume occupied by the gas, thick sprays where the volume fraction of the particles has to be considered together with the collision effects between particles. According to this classification, there are many possibilities for modelling sprays \cite{Des}. 

Recently, the interactions between particles and fluid have received increasing attention due to the number of their applications in the field of, for example, biotechnology, medicine, and in the study of sedimentation phenomenon, compressibility of droplets of the spray, cooling tower plumes, and diesel engines, etc \cite{BBJM, BDM, Rou, Will}. We refer to \cite{Rou} for more physical backgrounds of the spray model. Along with its applicative interest, the mathematical analysis of the proposed models or other mathematical modelling issues for real-world systems are also emphasized.

Among the various levels of possible description, in this paper, we are mainly focusing on the equations for moderately thick sprays which consists in coupled a kinetic equation of Vlasov-type with a fluid system. To be more precise, we are interested in the blow-up analysis for the coupled kinetic-fluid equations which consists of Vlasov equation and the isentropic compressible Navier-Stokes equations with density dependent viscosity coefficients where the coupling of two equations is through the drag force. 

Let $f(x,v,t)$ be the one-particle distribution function at $(x,v) \in {\R^d} \times \R^d$ and at time $t$, and let $\rho(x,t)$ and $u(x,t)$ be the fluid density and velocity, respectively. Then, the Vlasov/Navier-Stokes equations read as
\begin{align}\label{main_eq}
\begin{aligned}
&\pa_t f + v \cdot \nabla_x f + \nabla_v \cdot (D(\rho,u-v)f) = Q(f,f), \quad (x,v,t) \in {\R^d} \times \R^d \times \R_+,\cr
&\pa_t \rho + \nabla_x \cdot (\rho u) = 0,\cr
&\pa_t (\rho u) + \nabla_x \cdot (\rho u \otimes u) + \nabla_x p(\rho) -\nabla_x \cdot \mathbb{S}(\nabla_x u) = -\int_{\R^d} D(\rho,u-v)f\,dv,
\end{aligned}
\end{align}
with 
\[
\mathbb{S}(\nabla_x u) := 2\mu(\rho) \mathbb{T}(u) + \lambda(\rho)(\nabla_x \cdot u)\mathbb{I},
\]
together with the initial data
\bq\label{ini_main_eq}
(f(x,v,0), \rho(x,0), u(x,0)) =: (f_0(x,v),\rho_0(x),u_0(x)), \quad (x,v) \in {\R^d} \times \R^d.
\eq
Here $D=D(\rho,u-v)$ and $Q=Q(f,f)$ denote the drag force and the particle interaction operator, respectively, and $\mu=\mu(\rho)$ and $\lambda=\lambda(\rho)$ are the two Lam\'e viscosity coefficients depending on the fluid density $\rho$ which satisfy
\[
\mu \geq 0 \quad \mbox{and} \quad 2\mu + d\lambda \geq 0.
\]
The pressure law $p$ and the strain tensor $\mathbb{T}$ are given by
$$\begin{aligned}
p(\rho) = \rho^\gamma \quad \mbox{with} \quad \gamma > 1 \quad \mbox{and} \quad \mathbb{T}(u) = \frac12\lt( \nabla_x u + (\nabla_x u)^T\rt).
\end{aligned}$$

The study of the existence theory for the coupled kinetic-fluid equations is by now a well-established research topic. The motion of a solid particle suspension in a Stokes flow is studied in \cite{Ham}. The dispersed phase is modelled by a transport kinetic equation, with accretion induced by Stokes drag force and a gravitational field, and the resulting model is a Vlasov/Stokes system. The existence of global weak solutions to the system in a bounded domain with reflection boundary conditions is discussed. For the interactions with the incompressible or compressible Navier-Stokes equations, global existence of weak solutions in a bounded domain with different boundary conditions is obtained in \cite{BDGM,CCK,MV,Yu}. Global existence of strong solutions to the type of Vlasov/incompressible Navier-Stokes (resp. compressible Navier-Stokes) equations in a periodic spatial domain is established in \cite{BCHK0} (resp. \cite{BCHK}). In two dimensions, the global existence of weak and strong solutions is investigated in \cite{CL} without any smallness assumptions on the initial data. For the Vlasov-Fokker-Planck/Navier-Stokes(or Euler) system in the whole space, global weak and strong solutions are studied in \cite{CKL1} and global classical solutions near Maxwellians converging asymptotically to them are constructed in \cite{CDM,CKL2,DL}. We also refer to [15, 32] for the Cauchy problem of the Vlasov/inhomogeneous Navier-Stokes equations with a density dependent drag forcing term, $D = \rho(u-v)$, which is more physically relevant. Local smooth solutions to the Vlasov/compressible Euler equations with density dependent drag forcing term is also studied in \cite{BD}. For the moderately thick sprays case, i.e., Vlasov-Boltzmann/Navier-Stokes equations, local smooth solutions are discussed in \cite{Math}.

Despite these fruitful developments on the existence theory, to the best of author's knowledge, finite time blow up of solutions to Vlasov/Navier-Stokes equations has not been studied so far. It is worth mentioning that in the case of thin sprays, where the volume fraction of the particles is negligible compared to the one of the fluid, we can deal with the fluid equations as the dominant one in the whole coupled kinetic-fluid equations. Thus, our strategy for constructing blowing up solutions in a finite time is to use the blow up phenomena of the fluid part in \eqref{main_eq} together with careful understanding of the coupling between kinetic and fluid equations since it is known that the blow-up of classical solutions occurs under suitable assumptions for the compressible Navier-Stokes equations. However, the blow-up estimate for the coupled kinetic-fluid equations are not at all easy because the kinetic and fluid equations have different characteristic curves.

Some of the previous works on the blow up analysis for the compressible fluid equations which are the dominant part of the coupled system \eqref{main_eq} can be summarized as follows. The finite time blow up of $C^1$ solutions for the compressible Euler equations is studied in \cite{Sid} when the initial data is constant outside a bounded set and the initial fluid velocity is sufficiently large in some region. For the compressible Navier-Stokes equations, blow up result is obtained in \cite{Xin} which is generalized to the case with heat conduction for compactly supported initial density in \cite{CJ}. For non-compactly supported initial data, blow up estimates are provided in \cite{JWX, Roz} under certain decay assumptions on solutions at far fields. More recently, the finite time singularity formation of regular solutions for viscous compressible fluids without heat conduction is stuided in \cite{LPZ,XY} for the initial data with isolated mass group.

Our main contribution is to provide a simple idea of showing finite time blow up of classical solutions of the coupled system \eqref{main_eq} and related systems under suitable assumptions on the initial data and viscosity coefficients together with   {\bf (H1)}-{\bf (H3)}. Our strategy does not require the compactly supported initial data. Taking into account several physical quantities such as mass, momentum, momentum weight, momentum of inertia, and energy which will be appeared in Section \ref{sec_pre1}, we construct the finite time blowing up solutions. We also extend this strategy to different types of related systems; isentropic Navier-Stokes equations, two-phase fluid model consisting of pressureless Euler and isentropic Navier-Stokes equations, thick sprays model. If we assume that there is no particle interactions, i.e., $f \equiv 0$, the coupled system \eqref{main_eq} reduces to the isentropic compressible Navier-Stokes equations which is studied in \cite{BDG, MV2}. Very recently, the blow up analysis of the system \eqref{main_eq} with $f \equiv 0$ under some restriction on the initial data together with a particular choice of the viscosity coefficients $\mu(\rho) = \rho^\delta$ and $\lambda(\rho) = (\delta-1)\rho^\delta$ with $\delta > 1$(see {\bf (H3)} for the relation between $\mu$ and $\lambda$) is provided in \cite{JWX}. In the present paper, we deal with much larger class of admissible viscosity coefficients(see Theorem \ref{main_thm} and Remark \ref{main_rmk}). In this particular case $\mu(\rho) = \rho^\delta$, we refine the result in \cite{JWX} extending the regime for the exponent $\delta$(see Section \ref{sec_icns} for details). Our strategy is also applicable to the two-phase fluid model which can be derived from the kinetic-fluid equations by taking moments on the kinetic equation together with a mono-kinetic assumption(see Section \ref{sec_tpf}). A final extension is a finite time blow up of solutions for thick sprays \cite{BDM}. We stress that existence of solutions(even for local-in-time) for thick sprays is still a challenging open problem. 

The rest of this paper is organized as follows. In Section \ref{sec_pre}, we introduce physical quantities mentioned before and provide {\it a priori} estimates of energy, momentum of inertia, and total momentum. We also present our frameworks and main results on finite-time blow-up of classical solutions to the system \eqref{main_eq} depending on different conditions of viscosity coefficients. Section \ref{sec_pf} is devoted to give the details of the proof of our main results. Finally, Section \ref{sec_ext} generalizes our strategy to related systems, such as isentropic Naiver-Stokes equations, compressible pressureless Euler/isentropic Navier-Stokes equations, and thick sprays.

%%%%%%%%%%%%%%%%%%%%%%%%%%%%%%%%%%%%%%%%%%%%%%%%%%%%%%%%%%%%%%%%%%%%%%%%%%%%%%%%%5
%
%
%                        \section{Preliminaries and main results}
%
%
%%%%%%%%%%%%%%%%%%%%%%%%%%%%%%%%%%%%%%%%%%%%%%%%%%%%%%%%%%%%%%%%%%%%%%%%%%%%%%%%%
\section{Preliminaries and main results}\label{sec_pre}
In this section, we briefly discuss the energy estimate for the coupled system \eqref{main_eq} and useful estimates which provide relations between physical quantities mentioned in Introduction. Using these newly defined quantities, we present the main results of this paper.
\subsection{A priori estimates}\label{sec_pre1} In this subsection, we provide several useful estimates which will be crucially used to show the finite-time blow-up of classical solutions to the system \eqref{main_eq} later in Section \ref{sec_pf}.

We first introduce several physical quantities:\newline

\noindent $\bullet$ Mass.-
\[
m_\rho (t) := \int_{\R^d} \rho\,dx, \quad m_f(t) := \int_{\R^d \times \R^d} f\,dxdv.
\]
$\bullet$ Momentum.-
\[
M(t):=  \into \rho u\,dx + \intor f v \,dxdv =: M_\rho(t) + M_f(t). 
\]
$\bullet$ Momentum weight.-
\[
W(t):= \into \rho u \cdot x\,dx + \intor (x\cdot v) f\,dxdv =: W_\rho(t) + W_f(t).
\]
$\bullet$ Momentum of inertia.-
\[
I(t) := \frac12 \into \rho|x|^2\,dx + \frac12 \intor f|x|^2\,dxdv =: I_\rho(t) + I_f(t).
\]
$\bullet$ Total energy.-
\[
E(t):= \frac12 \into \rho|u|^2\,dx + \frac{1}{\gamma - 1}\into \rho^\gamma\,dx + \frac12 \intor f|v|^2\,dxdv=: E_k(t) + E_i(t) + E_f(t).
\]

Note that some of physical quantities above are introduced in \cite{Sid, Xin}. 
%\begin{align}\label{eq_phy}
%\begin{aligned}
%& m_\rho (t) := \int_{\R^d} \rho\,dx, \quad m_f(t) := \int_{\R^d \times \R^d} f\,dxdv \hspace{7.35cm} \mbox{(Mass)}\cr
%&M(t):=  \into \rho u\,dx + \intor f v \,dxdv =: M_\rho(t) + M_f(t) \hspace{4.8cm} \mbox{(Momentum)}\cr
%&W(t):= \into \rho u \cdot x\,dx + \intor (x\cdot v) f\,dxdv =: W_\rho(t) + W_f(t) \hspace{2.58cm} \mbox{(Momentum weight)}\cr
%&I(t) := \frac12 \into \rho|x|^2\,dx + \frac12 \intor f|x|^2\,dxdv =: I_\rho(t) + I_f(t) \hspace{2.5cm} \mbox{(Momentum of inertia)}\cr
%& E(t):= \frac12 \into \rho|u|^2\,dx + \frac{1}{\gamma - 1}\into \rho^\gamma\,dx + \frac12 \intor f|v|^2\,dxdv=: E_k(t) + E_i(t) + E_f(t) \hspace{0.38cm} \mbox{(Energy)}
%\end{aligned}
%\end{align}
We then provide energy estimates of solutions to the system \eqref{main_eq}.
\begin{lemma}\label{lem_energy}Let $(f,\rho,u)$ be a classical solution to the system \eqref{main_eq}-\eqref{ini_main_eq} in the interval $[0,T]$. Then we have
$$\begin{aligned}
&(i)  \,\,\,\,\frac{d}{dt} m_\rho(t) = \frac{d}{dt} m_f(t) = \frac{d}{dt} M(t)= 0,\cr
&(ii) \,\, \frac{d}{dt}E(t) + \int_{\R^d} \lt(2\mu(\rho) \mathbb{T}(u) : \mathbb{T}(u) + \lambda(\rho)|\nabla_x \cdot u|^2\rt)dx + \int_{\R^d \times \R^d} D(\rho,u-v) \cdot (u-v) f\,dxdv\cr
&\qquad \qquad \quad= \frac12\int_{\R^d \times \R^d} Q(f,f)|v|^2\,dxdv,
\end{aligned}$$
for all $t \in [0,T]$. Here $\mathbb{A}:\mathbb{B} = \sum_{i=1}^m\sum_{j=1}^n a_{ij} b_{ij}$ for $\mathbb{A} = (a_{ij}), \mathbb{B} = (b_{ij}) \in \R^{mn}$. 
\end{lemma}
\begin{proof} The conservation of mass is clearly obtained. For the conservation of total momentum, we easily find
\[
\frac{d}{dt}\int_{\R^d \times \R^d} v f\,dxdv = \int_{\R^d \times \R^d} D f\,dxdv \quad \mbox{and} \quad \frac{d}{dt}\int_{\R^d} \rho u\,dx = -\int_{\R^d \times \R^d} D f\,dxdv.
\]
This yields 
\[
\frac{d}{dt} M(t) =0.
\]
For the estimate of total energy, we first get from $\eqref{main_eq}_1$ that
\[
\frac12\frac{d}{dt}\int_{\R^d \times \R^d} |v|^2 f\,dxdv = \int_{\R^d \times \R^d} v \cdot D f\,dxdv + \frac12\int_{\R^d \times \R^d} Q(f,f)|v|^2\,dxdv.
\]
Next, it follows from $\eqref{main_eq}_3$ that
\bq\label{est_ee}
\frac12\frac{d}{dt}\int_{\R^d} \rho|u|^2dx = - \int_{\R^d} u \cdot \nabla_x p + \int_{\R^d} u \cdot \lt( \nabla_x \cdot \mathbb{S}(\nabla_x u) \rt)dx - \int_{\R^d} u \cdot D f\,dxdv.
\eq
We use the continuity equation $\eqref{main_eq}_2$ together with the pressure law $p = \rho^\gamma$ to get
\[
\int_{\R^d} u \cdot \nabla_x p\,dx = \frac{1}{\gamma - 1}\frac{d}{dt} \int_{\R^d} \rho^\gamma \,dx.
\]
We also find that the second term in the right hand side of \eqref{est_ee} can be estimate as
\[
\int_{\R^d} u \cdot \lt( \nabla_x \cdot (2\mu \mathbb{T}(u)) + \nabla_x (\lambda \nabla_x \cdot u)\rt)dx = -\int_{\R^d} 2\mu \mathbb{T}(u): \mathbb{T}(u) + \lambda |\nabla_x \cdot u|^2\,dx.
\]
where we used
\[
\intr u \cdot \nabla_x \cdot (2\mu \mathbb{T}(u))\,dx = - \sum_{i,j=1}^d \mu\pa_i u_j (\pa_i u_j + \pa_j u_i) = - \sum_{i,j=1}^d \frac{\mu}{2} (\pa_i u_j + \pa_j u_i)^2.
\]
Combining the above estimates, we conclude the desired result.
\end{proof}
Since the masses $m_\rho$ and $m_f$, and the total momentum $M$ are conserved in time, we denote $m_\rho(0)$, $m_f(0)$, and $M(0)$ by $m_\rho$, $m_f$, and $M$, respectively. Throughout this paper, we assume that the all initial physical quantities are bounded, i.e., $m_\rho, m_f, M, I_0:=I(0), W_0:= W(0)$, and $E_0 := E(0) < \infty$.

\subsection{Main results} In this part, we present our frameworks and main results for the finite time blow up of classical solutions to the system \eqref{main_eq}. 

Throughout this paper, we assume that the drag forcing term $D$, the particle interaction operator $Q$, and the viscosity coefficients $\mu,\lambda$ satisfy the following conditions:

\begin{itemize}
\item[{\bf (H1)}] The drag forcing term $D: \R \times \R^d \to \R^d$ satisfies
\[
D(x,y)\cdot y \geq 0 \quad \mbox{for} \quad (x,y) \in \R_+ \times \R^d.
\]
\item[{\bf (H2)}] The particle interaction operator $Q$ satisfies
\[
\int_{\R^d} Q(f,f)\,dv = \int_{\R^d} Q(f,f)v_i \,dv = 0 \quad \mbox{and} \quad \int_{\R^d \times \R^d} Q(f,f)|v|^2\,dvdx \leq 0, \quad i=1,\cdots,d.
\]
\item[{\bf (H3)}] The viscosity coefficients $\mu$ and $\lambda$ satisfy
\bq\label{as_lm}
\lambda(\rho) = 2\rho \mu^\prime (\rho) - 2\mu(\rho), \quad \mu(\rho) \geq 0.
\eq
\end{itemize}
We next briefly discuss the above assumptions. As mentioned in Introduction, most available literature \cite{BD, BDGM, CK, Ham, Math, MV, WY} for spray models deal with a linear drag force such as $D = u-v$ or a density dependent drag force $ D = \rho(u-v)$. It is clear that these types of drag forces satisfy the condition {\bf (H1)}. More generally, we can also consider the following form of drag forcing term:
\[
D(\rho,u-v) = h(\rho)|u-v|^{\beta - 1}(u-v) \quad \mbox{with} \quad \beta \geq 0,
\] 
where $h$ can be any non-negative scalar functions. 

For the assumption of the particle interaction operator $Q(f,f)$, we can choose the collision operator $Q_0(f,f)$ given by
\[
Q_0(f,g)(v) = \frac12\int_{\R^d \times \ms^{d-1}} B(\theta, |v-v_*|)(f' g'_* + f'_*g' - fg_* - f_*g) d\omega dv_*,
\]
where
\[
f = f(v), \quad f_* = f(v_*), \quad f' = f(v'), \quad f'_* = f(v'_*),
\]
\[
v' = v - [(v-v_*)\cdot \omega] \omega, \quad v'_* = v_* +  [(v-v_*)\cdot \omega] \omega, \quad \omega \in \ms^{d-1},
\]
$g$ is also similarly defined. $B(\theta, |v - v_*|)$ is the collision kernel depending only on $|v - v_*|$ and $\cos \theta$ with
\[
\cos \theta = \frac{v - v_*}{|v - v_*|}\cdot \omega > 0.
\]
Note that any collision invariant is a function of the form
\[
\phi(v) = a + \sum_{i=1}^d b_i v_i + c|v|^2, \quad a, b_1,\cdots, b_d,  c \in \R.
\]
In particular, the above collision operator $Q_0(f,f)$ satisfies {\bf (H2)}. In \cite{BCHK} and \cite{CCK}, the following forms of alignment forces in velocity are considered for the particle interactions, respectively:
\[
Q_1(f,f) := -\nabla_v \cdot (F_a(f)f) \quad \mbox{and} \quad Q_2(f,f) := -\nabla_v \cdot ((u_f - v)f),
\]
where $F_a(f)$ and $u_f$ represent global and local alignment forces, respectively, which are given by
\[
F_a(f)(x,v,t) := \int_{\R^d \times \R^d} \psi(x-y)(w-v)f(y,w,t)\,dydw \quad \mbox{and} \quad u_f(x,t) := \frac{\int_{\R^d} vf(x,v,t)\,dv}{\int_{\R^d} f(x,v,t)\,dv}.
\]
Here $\psi$ is a positive symmetric function called communication weights. For those particle interaction operators, we easily find that for $j=1,2$
\[
\int_{\R^d} Q_j(f,f)\,dv = \int_{\R^d} Q_j(f,f)v_i\,dv = 0, \quad i=1,\cdots,d,
\]
\[
\int_{\R^d \times \R^d} Q_1(f,f)|v|^2\,dxdv = -\int_{\R^d \times \R^d} \psi(x-y)|w-v|^2 f(x,v)f(y,w)\,dxdydvdw \leq 0,
\]
and
\[
\int_{\R^d \times \R^d} Q_2(f,f)|v|^2\,dxdv = -2\int_{\R^d \times \R^d} |u_f - v|^2 f\,dxdv \leq 0.
\]

The relation between the viscosity coefficients $\mu$ and $\lambda$ presented in {\bf (H3)} is fundamental to obtain more regularity on the fluid density $\rho$. We refer to \cite{BDe} for more details. If we choose $\mu(\rho) = \rho^\delta$ with $\gamma \geq \delta \geq 1$, then $\lambda(\rho) = 2(\delta - 1)\rho^\delta = 2(\delta-1)\mu(\rho)$. Moreover, if $\rho \in L^\infty(0,T;(L^1\cap L^\gamma)(\R^d))$, then we use the interpolation inequality to get
\bq\label{est_mu}
\int_{\R^d} \mu(\rho)\,dx = \int_{\R^d} \rho^\delta\,dx \leq \lt( \int_{\R^d} \rho\,dx\rt)^{\frac{\gamma - \delta}{\gamma - 1}}\lt(\int_{\R^d} \rho^\gamma \,dx\rt)^{\frac{\delta - 1}{\gamma - 1}},
\eq
and, subsequently, this implies $\mu(\rho) \in L^\infty(0,T;L^1(\R^d))$. We also notice that if $\delta=1$, i.e., $\mu(\rho) = \rho$ and $\lambda(\rho) = 0$, the fluid part of the system \eqref{main_eq} becomes the viscous Saint-Venant system for the shallow water. 

Before stating our main results, we define a solution space $\mathfrak{S}_0$ as follows.
\begin{definition} For any $T>0$, we call $(f,\rho,u) \in \mathfrak{S}_0(T)$ if $(f,\rho,u)$ is a classical solution to the Cauchy problem \eqref{main_eq}-\eqref{ini_main_eq} in the time interval $[0,T]$ satisfying the following conditions of decay at far fields:
\[
f|v|\lt( |x|^2 + |v|^2\rt) \to 0 \quad \mbox{as} \quad |x| \to \infty, \quad \forall \,(v,t) \in \R^d \times [0,T],
\]
\[
f|D(\rho,u-v)|\lt( |x|^2 + |v|^2\rt) \to 0 \quad \mbox{as} \quad |v| \to \infty, \quad \forall \,(x,t) \in \R^d \times [0,T],
\]
\[
\mu(\rho)|u| \to 0, \quad \lt(p(\rho) + \lt( \mu(\rho) + \lambda(\rho)\rt)|\nabla_x u| \rt)\lt( |x| + |u|\rt) \to 0 \quad \mbox{as} \quad |x| \to \infty, \quad \forall \,t \in [0,T],
\]
and
\[
\rho|u|\lt(|x|^2 + |u|^2 + |u||\nabla_x u| \rt) \to 0 \quad \mbox{as} \quad |x| \to \infty, \quad \forall \,t \in [0,T].
\]
\end{definition}
We notice that the decay condition for solutions allows us to do the integration by parts in our estimates. 

\begin{remark} If we consider the compactly supported particle density $f$, then we only need the following conditions for the classical solutions:
\[
\mu(\rho)|u| \to 0, \quad \lt(p(\rho) + \lt( \mu(\rho) + \lambda(\rho)\rt)|\nabla_x u| \rt)\lt( |x| + |u|\rt) \to 0 \quad \mbox{as} \quad |x| \to \infty, \quad \forall \,t \in [0,T],
\]
and
\[
\rho|u|\lt(|x|^2 + |u|^2 + |u||\nabla_x u| \rt) \to 0\quad \mbox{as} \quad |x| \to \infty, \quad \forall \,t \in [0,T].
\]
Note that the above conditions are required for the blow-up estimates of isentropic Navier-Stokes equations(see Definition \ref{def_isenNS}).
\end{remark}

We are now in a position to address our main results on the finite time blow up of classical solutions to the system \eqref{main_eq} in this paper:

\begin{theorem}\label{main_thm}Let $(f,\rho,u)$ be a solution to the Cauchy problem to \eqref{main_eq}-\eqref{ini_main_eq} satisfying $(f,\rho,u) \in \mathfrak{S}_0(T)$. Suppose that $1 < \gamma < 1 + \frac{1}{d}$ and the viscosity coefficient $\mu$ and $\lambda$ satisfy $\mu(\rho),\, \lambda(\rho) \in L^\infty(0,T;L^1(\R^d))$, and there exist two positive constants $M_\mu, M_\lambda$ which are independent of $T$ such that
\bq\label{as_lm22}
\|\mu(\rho)\|_{L^\infty(0,T; L^1(\R^d))} \leq M_\mu, \qquad  \|\lambda(\rho)\|_{L^\infty(0,T; L^1(\R^d))} \leq M_\lambda.
\eq
Furthermore, the initial data \eqref{ini_main_eq} satisfy
\[
C_0 > \frac12\lt(\frac{M_\mu}{2} + \max\{2,d(\gamma - 1)\}E_0\rt) \lt(J_0 + \frac{M_\mu/2 + dM_\lambda/4}{1 - d(\gamma-1)}\rt),
\]
where $C_0$ and $J_0$ are nonnegative constants given by
\[
C_0 := \lt(\frac{\pi^{d/2}}{\Gamma\lt(d/2 +1 \rt)}\rt)^{1 - \gamma}\frac{m_\rho^{\frac{(d+2)\gamma - d}{2}}}{2^{\frac{(d+2)\gamma - d}{2}}(\gamma-1)},
\]
and $J_0 := I_0 - W_0 + E_0 \geq 0$, respectively. Here $\Gamma$ is the gamma function. Then the life-span $T$ of the solution $(f,\rho,u)$ is finite.
\end{theorem}
\begin{remark}\label{main_rmk} If $\mu(\rho) = \rho$, then $\lambda(\rho) = 0$ due to the relation in {\bf (H3)}. In this case, we can choose $M_\mu = m_\rho$ and $M_\lambda = 0$ in \eqref{as_lm22}. It follows from \eqref{est_mu} that any linear combination of the form $\sum c_k \rho^{n_k}$ with $c_k \geq 0, n_k \in [1,\gamma]$ for all $k$ is an admissible function for $\mu(\rho)$.
\end{remark}

\begin{theorem}\label{main_thm2}Let $(f,\rho,u)$ be a solution to the Cauchy problem to \eqref{main_eq}-\eqref{ini_main_eq} satisfying $(f,\rho,u) \in \mathfrak{S}_0(T)$. Suppose that the viscosity coefficient $\mu$ has a form of $\mu(\rho) = \rho^\delta$ with $\delta \in (1,\gamma]$.
\begin{itemize} 
\item If $1 < \gamma < 1 + \frac{2}{d}$, $\delta = \gamma$, and the initial data satisfy
\[
C_0 > \frac{J_0 e^{(\gamma-1)\lt(1 + d(\delta-1) \rt)/2}}{2}\lt(\frac{M_\mu}{2} + \max\{2,d(\gamma - 1)\}E_0\rt).
\]
\item If $1 < \gamma < 1 + \frac{2}{d}$, $\gamma - \frac1d< \delta < \gamma$ and the initial data satisfy
\[
C_0 >\frac12\lt(\frac{M_\mu}{2} + \max\{2,d(\gamma - 1)\}E_0\rt)\lt(J_0^{\frac{\gamma-\delta}{\gamma-1}} +\frac{\lt(1 + d(\delta-1) \rt) \,m_\rho^{\frac{\gamma-\delta}{\gamma-1}}(\gamma-1)^{\frac{\delta-\gamma}{\gamma-1}}(\gamma-\delta)}{2(1 - d(\gamma - \delta))}\rt)^{\frac{\gamma-1}{\gamma-\delta}}.
\]
\end{itemize}
Here $C_0$, $M_\mu$, and $J_0$ are given as in Theorem \ref{main_thm}. Then the life-span $T$ of the solution $(f,\rho,u)$ is finite.
\end{theorem}

\begin{remark}The constant $M_\mu > 0$ can estimated by the initial mass $m_\rho$ and total energy $E_0$ due to the explicit form of viscosity coefficient $\mu$ together with the relation \eqref{as_lm}$($see \eqref{est_sp}$)$.
\end{remark}

%\begin{remark}\label{rmk_gd}If we set $f \equiv 0$, then the system \eqref{main_eq} becomes the barotropic compressible Navier-Stokes equations. In a recent work \cite{JWX}, the blow up estimate for this system when $\gamma \in (1,1 + \frac{2}{d})$ and $\frac{\gamma + 1}{2} < \delta \leq \gamma$ together with a suitable assumption on the initial data. We note that the strategy used for Theorem \ref{main_thm2} can be directly applied to the case $f \equiv 0$. Since $\gamma - \frac1d < \frac{\gamma + 1}{2}$ if and only if $\gamma < 1 + \frac2d$, this extends the blow-up estimate obtained in \cite{JWX}.
%\end{remark}

As a simple consequence of Theorem \ref{main_thm2}, we provide a blow up analysis for the system \eqref{main_eq} with constant viscosity coefficients:

\begin{corollary}\label{main_coro}Let $(f,\rho,u)$ be a solution satisfying $(f,\rho,u) \in \mathfrak{S}_0(T)$. Suppose that $1 < \gamma \leq 1 + \frac{2}{d}$ and the initial data satisfy $C_0 > (\max\{1, d(\gamma - 1)/2\} E_0)^{\frac{d(\gamma-1)}{2}} J_0$, where $C_0$ and $J_0$ are given as in Theorem \ref{main_thm}.
Then the life-span $T$ of the solution $(f,\rho,u)$ is finite.
\end{corollary}
\begin{remark}\label{rem_1}Our strategy of the blow-up estimates also holds for the Euler equations instead of the barotropic Navier-Stokes equations in the system \eqref{main_eq}, i.e., $\mu = \lambda = 0$.
\end{remark}
\begin{remark}The large-time behavior for the system \eqref{main_eq} with the density independent drag force $D = u-v$ and constant viscosity coefficients in periodic spatial domain $\T^d, d \geq 3$ is studied in \cite{Ch2}.
\end{remark}

%%%%%%%%%%%%%%%%%%%%%%%%%%%%%%%%%%%%%%%%%%%%%%%%%%%%%%%%%%%%%%%%%%%%%%%%%%%%%%%%%5
%
%
%                        Section: Finite-time blow-up of classical solutions
%
%
%%%%%%%%%%%%%%%%%%%%%%%%%%%%%%%%%%%%%%%%%%%%%%%%%%%%%%%%%%%%%%%%%%%%%%%%%%%%%%%%%
\section{Finite-time blow-up of classical solutions}\label{sec_pf}
In this section, we provide the details of the proof of Theorems \ref{main_thm}, \ref{main_thm2}, and Corollary \ref{main_coro}. For this, we first estimate the total momentum of inertia and the lower bound of the kinetic energies in the lemma below.

\begin{lemma}\label{lem_conserv}Let $(f,\rho,u)$ be a classical solution to the system \eqref{main_eq}-\eqref{ini_main_eq} in the interval $[0,T]$. Then it holds
\[
\frac{d}{dt}I_\rho(t) = W_\rho(t), \quad \frac{d}{dt}I_f(t) = W_f(t), \quad \mbox{i.e.,} \quad \frac{d}{dt} I(t) = W(t),
\]
and
\[
M^2 \leq 4\max\{m_\rho, m_f\}\lt(E_k(t) + E_f(t)\rt),
\]
for $t \in [0,T]$.
\end{lemma}
\begin{proof}We can easily check 
\[
\frac{d}{dt}I_\rho(t) = W_\rho(t) \quad \mbox{and} \quad \frac{d}{dt}I_f(t) = W_f(t).
\]
Using H\"older inequality together with Cauchy-Schwartz inequality, we get
\bq\label{est_ee}
M(t)^2 \leq 2M_\rho(t)^2 + 2M_f(t)^2 \leq 4m_\rho E_k(t) + 4m_f E_f(t) \leq 4\max\{m_\rho,m_f\}\lt(E_k(t) + E_f(t) \rt).
\eq
Then we conclude the desired inequality by using the conservation of total momentum.
\end{proof}

We next estimate the upper bound of the total momentum of inertia and the lower bound of the function $E_i$.

\begin{lemma}\label{lem_mww2} The total momentum of inertia is bounded by a quadratic function of $t$:
\[
I(t) \leq I_0 + C_1t + C_2t^2,
\]
where $C_1$ and $C_2$ are positive constants given by
\[
C_1:= W_0 + E_0 + 2dM_\mu \quad \mbox{and} \quad C_2 :=\frac12\lt(\frac{M_\mu}{2} + \max\{2,d(\gamma - 1)\}E_0\rt),
\]
respectively.
\end{lemma}
\begin{proof}A straightforward computation yields
$$\begin{aligned}
\frac{d}{dt} W(t) &= \int_{\R^d} \rho |u|^2\,dx + d\int_{\R^d} p(\rho)\,dx + \int_{\R^d \times \R^d} |v|^2 f\,dxdv - \int_{\R^d} (2\mu(\rho) + d\lambda(\rho))(\nabla_x \cdot u)\,dx\cr
&=2(E_f + E_k) + d(\gamma - 1) E_i- \int_{\R^d} (2\mu(\rho) + d\lambda(\rho))(\nabla_x \cdot u)\,dx.
\end{aligned}$$
%\textcolor{blue}{More details..}
On the other hand, it follows from $\eqref{main_eq}_2$ that
$$\begin{aligned}
\frac{d}{dt}\int_{\R^d} \mu(\rho)\,dx &= \int_{\R^d} \mu^\prime(\rho) \rho_t\,dx = -\int_{\R^d} \mu^\prime(\rho)\nabla_x \cdot(\rho u)\,dx = -\int_{\R^d} (\rho \mu^\prime(\rho) - \mu(\rho))(\nabla_x \cdot u)dx,
\end{aligned}$$
where we used
\[
-\int_{\R^d} \mu^\prime(\rho) \nabla_x \rho \cdot u\,dx = -\int_{\R^d} \nabla_x \mu(\rho) \cdot u\,dx = \int_{\R^d} \mu(\rho) (\nabla_x \cdot u)\,dx.
\]
We remaind the reader that the viscosity coefficients $\mu$ and $\lambda$ satisfy $\lambda(\rho) = 2\rho\mu^\prime(\rho) - 2\mu(\rho)$. Thus we obtain
\bq\label{est_mul}
\frac{d}{dt}\int_{\R^d} \mu(\rho)\,dx = -\frac12\int_{\R^d} \lambda(\rho)(\nabla_x \cdot u)\,dx,
\eq
This and together with Lemma \ref{lem_energy} implies
\begin{align}\label{est_w1}
\begin{aligned}
\frac{d}{dt} W(t) &= 2(E_f + E_f) + d(\gamma - 1) E_i- \int_{\R^d} 2\mu(\rho)(\nabla_x \cdot u)\,dx + 2d\frac{d}{dt}\int_{\R^d} \mu(\rho)\,dx\cr
&\leq \max\{2,d(\gamma - 1)\}E_0 - \int_{\R^d} 2\mu(\rho)(\nabla_x \cdot u)\,dx + 2d\frac{d}{dt}\int_{\R^d} \mu(\rho)\,dx.
\end{aligned}
\end{align}
Using the H\"older inequality, the energy estimate in Lemma \ref{lem_energy}, and the assumption \eqref{as_lm22}, we estimate
\begin{align}\label{est_w2}
\begin{aligned}
\int_{\R^d} 2\mu(\rho)(\nabla_x \cdot u)\,dx &\leq \frac12\int_{\R^d} \mu(\rho)\,dx + \int_{\R^d} 2\mu(\rho)|\nabla_x \cdot u|^2\,dx \cr
&\leq \frac{M_\mu}{2} + \int_{\R^d} 2\mu(\rho)\mathbb{T}(u):\mathbb{T}(u)\,dx \cr
&\leq \frac{M_\mu}{2} - \frac{d}{dt}E(t).
\end{aligned}
\end{align}
Combining \eqref{est_w1} and \eqref{est_w2}, we obtain
\[
\frac{d}{dt} \lt(W(t) + E(t) \rt) \leq \max\{2,d(\gamma - 1)\}E_0 + \frac{M_\mu}{2}  + 2d\frac{d}{dt}\int_{\R^d} \mu(\rho)\,dx.
\]
We now integrate the above inequality over the time interval $[0,t]$ and use the uniform boundedness of $\mu(\rho)$ \eqref{as_lm22} again to get
\[
W(t) + E(t) \leq W_0 +  E_0 + 2dM_\mu + \lt(\frac{M_\mu}{2} + \max\{2,d(\gamma - 1)\}E_0 \rt)t
\]
Finally we integrate the above inequality over $[0,t]$ to conclude
\[
I(t) \leq I_0 + \lt( W_0 +  E_0 + 2dM_\mu\rt)t + \frac12\lt(\frac{M_\mu}{2} + \max\{2,d(\gamma - 1)\}E_0\rt)t^2,
\]
due to $E(t) \geq 0$ and $I^\prime(t) = W(t)$.
\end{proof}

The following inequality gives the relation between $E_i$ and $I_\rho$.

\begin{lemma}\label{lem_mw3}There exists a positive constant $C_0$ such that
\[
E_i(t) \geq \frac{C_0}{I_\rho(t)^{\frac{d(\gamma-1)}{2}}}.
\]
Here $C_0 > 0$ is given as in Theorem \ref{main_thm}.
\end{lemma}
\begin{proof}
%The proof is similar to the one in \cite{JWX}. 
For any $R>0$, we estimate $m_\rho$ as
$$\begin{aligned}
\intr \rho\,dx &= \int_{|x| \leq R} \rho\,dx + \int_{|x| \geq R} \rho\,dx \cr
&\leq |B(0,R)|^{1-\frac1\gamma}\lt(\int_{|x| \leq R}\rho^\gamma\,dx \rt)^{\frac1\gamma} + \frac{1}{R^2} \int_{|x| \geq R}\rho |x|^2\,dx\cr
&\leq |B(0,1)|^{1-\frac1\gamma}R^{d(1 - \frac1\gamma)}\lt(\intr \rho^\gamma\,dx \rt)^{\frac1\gamma} + \frac{1}{R^2} \intr \rho |x|^2\,dx,
\end{aligned}$$
where $B(0,R) := \{ x \in \R^d : |x| \leq R\}$, and $|A|$ denotes the Lebesgue measure of a set $A$ in $\R^d$. Here we used H\"older's inequality. Then we choose $R$ so that the right hand side of the above relation is minimized as
\[
R = \lt(\frac{\| |x|^2\rho\|_{L^1}}{\|\rho\|_{L^\gamma}|B(0,1)|^{1 - \frac1\gamma} }\rt)^{\frac{\gamma}{(d+2)\gamma - d}}.
\]
Since $|B(0,1)| = \pi^{d/2}/\Gamma(d/2 + 1)$, this gives the desired inequality.
\end{proof}

%%%%%%%%%%%%%%%%%%%%%%%%%%%%%%%%%
%
%   \subsection{Proof of Theorem \ref{main_thm}}
%
%%%%%%%%%%%%%%%%%%%%%%%%%%%%%%%%%

\subsection{Proof of Theorem \ref{main_thm}} Inspired by \cite{Xin}, we estimate the upper bound of $E_i$. For this, we set
\[
J(t) := I(t) - (t+1)W(t) + (t+1)^2E(t).
\]
In the lemma below, we provide the upper bound of $J$ which will directly give us the upper bound of $E_i$.
\begin{lemma}\label{lem_mww4}For $1 < \gamma < 1 + \frac1d$, we have the following upper bounds of $J$:
\[
J(t) \leq \lt(J_0 + \frac{M_\mu/2 + dM_\lambda/4}{1 - d(\gamma-1)}\rt)(t+1)^{2 - d(\gamma - 1)}.
\]
\end{lemma}
\begin{proof}We first decompose the function $J$ into two terms: $J = J_\rho + J_f$ where
\[
J_\rho(t) := I_\rho(t) - (t+1)W_\rho(t) + (t+1)^2E_\rho(t) \quad \mbox{and} \quad J_f(t) := I_f(t) - (t+1)W_f(t) + (t+1)^2E_f(t).
\]
Here $E_\rho := E_k - E_i$. Then since
\[
W_\rho(t)^2 \leq 4I_\rho(t) E_k(t) \quad \mbox{and} \quad W_f(t)^2\leq 4I_f(t) E_f(t),
\]
we get
\[
I_\rho(t) - (t+1)W_\rho(t) + (t+1)^2E_k(t) \geq 0 \quad \mbox{and} \quad I_f(t) - (t+1)W_f(t) + (t+1)^2E_f(t) \geq 0.
\]
This implies
\bq\label{est_j}
J_\rho(t) \geq (t+1)^2 E_i(t), \quad J_f(t) \geq 0, \quad \mbox{and} \quad J(t) \geq (t+1)^2 E_i(t).
\eq
On the other hand, it follows from Lemma \ref{lem_conserv} that we find
\[
\frac{d}{dt} J_\rho(t) = -(t+1)\frac{d}{dt} W_\rho(t) + 2(t+1)E_\rho(t) + (t+1)^2 \frac{d}{dt} E_\rho(t),
\]
and
\[
\frac{d}{dt} J_f(t) = -(t+1)\frac{d}{dt} W_f(t) + 2(t+1)E_f(t) + (t+1)^2 \frac{d}{dt}E_f(t),
\]
Then this and together with the identity \eqref{est_mul} yields

\begin{align}\label{est_jj}
\begin{aligned}
\frac{d}{dt}J(t) &= -(t+1)\frac{d}{dt}W(t) + 2(t+1)E(t) + (t+1)^2\frac{d}{dt}E(t)\cr
&= (t+ 1)(2 - d(\gamma - 1))E_i + (t+1)\int_{\R^d} (2\mu(\rho) + d\lambda(\rho))(\nabla_x \cdot u)\,dx + (t+1)^2\frac{d}{dt} E(t).
%&\leq (t+1)(2 - d(\gamma - 1))E_i + (\nu + 2d)\int_{\R^d} \mu(\rho)\,dx -2d\frac{d}{dt}\lt( (t+1)\int_{\R^d} \mu(\rho)\,dx\rt),
\end{aligned}
\end{align}
On the other hand, we use the H\"older's inequality to find
\begin{align}\label{est_jjj}
\begin{aligned}
2(t+1)\int_{\R^d} \mu(\rho)(\nabla_x \cdot u)\,dx &\leq \frac12\int_{\R^d} \mu(\rho)\,dx + (t+1)^2\int_{\R^d} 2\mu(\rho)|\nabla_x \cdot u|^2\,dx\cr
&\leq \frac{M_\mu}{2}  + (t+1)^2\int_{\R^d} 2\mu(\rho)\mathbb{T}(u):\mathbb{T}(u)\,dx\cr
d(t+1)\int_{\R^d} \lambda(\rho)(\nabla_x \cdot u)\,dx & \leq \frac d4\intr \lambda(\rho)\,dx + (t+1)^2\int_{\R^d} \lambda(\rho)|\nabla_x \cdot u|^2\,dx\cr
&= \frac{dM_\lambda}{4} + (t+1)^2\int_{\R^d} \lambda(\rho)|\nabla_x \cdot u|^2\,dx.
\end{aligned}
\end{align}
This and together with the energy estimate in Lemma \ref{lem_energy} implies
$$\begin{aligned}
&(t+1)\int_{\R^d} (2\mu(\rho) + d\lambda(\rho))(\nabla_x \cdot u)\,dx \cr
&\quad \leq \frac{M_\mu}{2} + \frac{dM_\lambda}{4} + (t+1)^2\lt(\int_{\R^d} 2\mu(\rho)\mathbb{T}(u):\mathbb{T}(u)\,dx+ \int_{\R^d} \lambda(\rho)|\nabla_x \cdot u|^2\,dx\rt)\cr
&\quad \leq \frac{M_\mu}{2} + \frac{dM_\lambda}{4} -(t+1)^2\frac{d}{dt} E(t).
\end{aligned}$$
Thus we get
\bq\label{diff_j}
\frac{d}{dt}J(t) \leq (t+ 1)(2 - d(\gamma - 1))E_i + \frac{M_\mu}{2} + \frac{dM_\lambda}{4} \leq \frac{2 - d(\gamma - 1)}{(t+1)}J(t) + \frac{M_\mu}{2} + \frac{dM_\lambda}{4},
\eq
due to the estimate of lower bound of $J$ \eqref{est_j}. Applying the Gronwall's inequality to \eqref{diff_j}, we find
\[
\frac{J(t)}{(t+1)^{2 - d(\gamma - 1)}} \leq J_0 + \frac{M_\mu/2 + dM_\lambda/4}{1 - d(\gamma-1)}\lt(1 - \frac{1}{(t+1)^{1-d(\gamma-1)}} \rt),
\]
that is, $J$ satisfies the following inequality:
$$\begin{aligned}
J(t) &\leq J_0(t+1)^{2 - d(\gamma - 1)} + \frac{M_\mu/2 + dM_\lambda/4}{1 - d(\gamma-1)}\lt((t+1)^{2 - d(\gamma - 1)} - (t+1) \rt)\cr
&\leq \lt(J_0 + \frac{M_\mu/2 + dM_\lambda/4}{1 - d(\gamma-1)}\rt)(t+1)^{2 - d(\gamma - 1)}.
\end{aligned}$$
Here we used the assumption $1 - d(\gamma-1) > 0$. This completes the proof.
%$$\begin{aligned}
%\frac{d}{dt}\lt(J(t) +2d (t+1)\int_{\R^d} \mu(\rho)\,dx\rt) &\leq (t+ 1)(2 - d(\gamma - 1))E_i + (\nu + 2d)\int_{\R^d} \mu(\rho)\,dx\cr
%&\leq (t+ 1)(2 - d(\gamma - 1))E_i + (\nu + 2d)M_\mu.
%\end{aligned}$$
%We integrate the above inequality over $[0,t]$ to find
%\bq\label{est_j3}
%J(t) \leq J_0 + 2d\int_{\R^d} \mu(\rho_0)\,dx + \int_0^t \frac{2-d(\gamma - 1)}{s+1}J(s)\,ds + (\nu + 2d)M_\mu t,
%\eq
%due to \eqref{est_j} and the positiveness of $\mu$. In order to solve \eqref{est_j3}, we set
%\[
%F:= J_0 + 2dM_\mu+ \int_0^t \frac{2-d(\gamma - 1)}{s+1}J(s)\,ds + (\nu + 2d)M_\mu t,
%\]
%then $F$ satisfies
%\[
%F^\prime(t) \leq \frac{2 - d(\gamma -1)}{t+1}F(t) + (\nu + 2d)M_\mu \quad \mbox{with} \quad F(0) = J_0 + 2dM_\mu,
%\]
%since $2 - d(\gamma - 1) > 1$. We now solve the above differential inequality as $$\begin{aligned}
%F(t) &\leq F(0)(t+1)^{2 - d(\gamma -1)} + \frac{(\nu + 2d)M_\mu}{1 - d(\gamma -1)}\lt((t+1)^{2 - d(\gamma - 1)} - (t+1) \rt)\cr
%&\leq F(0)(t+1)^{2 - d(\gamma -1)} + \frac{(\nu + 2d)M_\mu}{1 - d(\gamma -1)}(t+1)^{2 - d(\gamma - 1)}\cr
%&\leq \max\lt\{F(0), \frac{(\nu + 2d)M_\mu}{1 - d(\gamma -1)}\rt\}(t+1)^{2 - d(\gamma -1)}.
%\end{aligned}$$
%This yields
%\[
%J(t) \leq \max\lt\{J_0 + 2dM_\mu,\, \frac{(\nu + 2d)M_\mu}{1 - d(\gamma -1)}\rt\}(t+1)^{2 - d(\gamma -1)}.
%\]
\end{proof}

\begin{proof}[Proof of Theorem \ref{main_thm}] For $2 - d(\gamma-1) > 1$, it follows from Lemmas \ref{lem_mw3} and \ref{lem_mww4} that
\[
\frac{C_0}{I_\rho(t)^{\frac{d(\gamma-1)}{2}}} \leq E_i(t) \leq \frac{C_3}{(t+1)^{d(\gamma - 1)}},
\]
where $C_3$ is a positive constant given by
\[
C_3 := J_0 + \frac{M_\mu/2 + dM_\lambda/4}{1 - d(\gamma-1)}.
\]
On the other hand, we also find from Lemma \ref{lem_mww2} that  
\[
I_\rho(t) \leq I(t) \leq I_0 + C_1 \, t + C_2\, t^2,
\]
and this yields
\[
\frac{C_0}{(I_0 + C_1 \, t + C_2 \, t^2)^{\frac{d(\gamma-1)}{2}}} \leq E_i(t).
\]
Thus we obtain
\[
\frac{C_0}{(I_0 + C_1 \, t + C_2 \, t^2)^{\frac{d(\gamma-1)}{2}}} \leq \frac{C_3}{(t+1)^{d(\gamma - 1)}}.
\]
Note that as $t \to +\infty$ the above inequality implies
\[
C_0 \leq C_2 C_3.
\]
This concludes that the life-span $T$ of classical solutions to the system \eqref{main_eq}-\eqref{ini_main_eq} with the initial data satisfying $C_0 > C_2 C_3$ should be finite.
\end{proof}

%%%%%%%%%%%%%%%%%%%%%%%%%%%%%%%%%
%
%   \subsection{Proof of Theorem \ref{main_thm2}}
%
%%%%%%%%%%%%%%%%%%%%%%%%%%%%%%%%%

\subsection{Proof of Theorem \ref{main_thm2}}

In the lemma below, we consider the viscosity coefficient $\mu$ has the form of $\rho^\delta$ with $\gamma \geq \delta > 1$. We notice that if $\mu(\rho) = \rho^\delta$, then $\mu^\prime \geq 0$ and $\lambda(\rho) = 2(\delta - 1)\rho^\delta$. Thus we can choose $M_\lambda = 2(\delta-1)M_\mu$ in this case.
\begin{lemma}\label{lem_gd} For $1 < \gamma < 1 + \frac2d$, we have the following upper bounds of $J$:
\[
J(t) \leq \left\{ \begin{array}{ll}
J_0 e^{\nu(\gamma-1)}(t+1)^{2 - d(\gamma - 1)} & \textrm{if $\delta = \gamma$,}\\[1mm]
\lt(J_0^{\frac{\gamma-\delta}{\gamma-1}} +\frac{\nu \,m_\rho^{\frac{\gamma-\delta}{\gamma-1}}(\gamma-1)^{\frac{\delta-\gamma}{\gamma-1}}(\gamma-\delta)}{1 - d(\gamma - \delta)}\rt)^{\frac{\gamma-1}{\gamma-\delta}}(t+1)^{2 - d(\gamma-1)} & \textrm{if $\gamma > \delta > \gamma - \frac1d$,}
  \end{array} \right.
\]
where $\nu := \lt(1 + d(\delta-1) \rt)/2$.
\end{lemma}
\begin{proof}Using the same argument as in the proof of Lemma \ref{lem_mww4}, we first obtain \eqref{est_j}. Then we refine the estimate of $J$ as
$$\begin{aligned}
\frac{d}{dt} J(t) &=(t+ 1)(2 - d(\gamma - 1))E_i + (t+1)\int_{\R^d} (2\mu(\rho) + d\lambda(\rho))(\nabla_x \cdot u)\,dx + (t+1)^2\frac{d}{dt} E(t)\cr
&\leq (t+1)(2 - d(\gamma - 1))E_i(t) + \frac12\int_{\R^d} \mu(\rho)\,dx + \frac d4\intr \lambda(\rho)\,dx\cr
& = (t+1)(2 - d(\gamma - 1))E_i(t) + \frac12\lt(1 + d(\delta-1) \rt)\int_{\R^d} \mu(\rho)\,dx\cr
&\leq (t+1)(2 - d(\gamma - 1))E_i(t) + \frac{m_\rho^{\frac{\gamma-\delta}{\gamma-1}}(\gamma-1)^{\frac{\delta-1}{\gamma-1}}\lt(1 + d(\delta-1) \rt)}{2}E_i(t)^{\frac{\delta-1}{\gamma-1}},
\end{aligned}$$
where we used \eqref{as_lm}, \eqref{est_mu}, \eqref{est_jj}, and \eqref{est_jjj}. If $2 - d(\gamma - 1) > 0$, then it follows from \eqref{est_j} that
\[
\frac{d}{dt}J(t) \leq \frac{2 - d(\gamma - 1)}{(t+1)}J(t) + \frac{m_\rho^{\frac{\gamma-\delta}{\gamma-1}}(\gamma-1)^{\frac{\delta-1}{\gamma-1}}\lt(1 + d(\delta-1) \rt)}{2(t+1)^{\frac{2(\delta - 1)}{\gamma-1}}}J(t)^{\frac{\delta-1}{\gamma-1}}.
\]
We now apply Lemma \ref{lem_usef2} to the above differential inequality with 
\[
a = 2 - d(\gamma - 1), \quad b = \frac{m_\rho^{\frac{\gamma-\delta}{\gamma-1}}(\gamma-1)^{\frac{\delta-1}{\gamma-1}}\lt(1 + d(\delta-1) \rt)}{2},\quad \mbox{and}\quad\beta = \frac{\delta-1}{\gamma-1},
\]
to have the following inequalities:
\begin{itemize}
\item If $\delta = \gamma$, i.e., $b = (\gamma-1)\lt(1 + d(\delta-1) \rt)/2$ and $\beta = 1$, then we have
\[
J(t) \leq J_0 e^{(\gamma-1)\lt(1 + d(\delta-1) \rt)/2}(t+1)^{2 - d(\gamma - 1)}.
\]
\item If $\delta > \gamma - \frac1d$, then it is clear to get $2\beta + a(1-\beta) = 2-d(\gamma-\delta) > 1$, and this deduces
\[
J(t) \leq \lt(J_0^{\frac{\gamma-\delta}{\gamma-1}} +\frac{\lt(1 + d(\delta-1) \rt) \,m_\rho^{\frac{\gamma-\delta}{\gamma-1}}(\gamma-1)^{\frac{\delta-\gamma}{\gamma-1}}(\gamma-\delta)}{2(1 - d(\gamma - \delta))}\rt)^{\frac{\gamma-1}{\gamma-\delta}}(t+1)^{2 - d(\gamma-1)}.
\]
\end{itemize}
This completes the proof.
\end{proof}

\begin{proof}[Proof of Theorem \ref{main_thm2}] We first notice that for $\gamma \geq \delta > 1$
\[%\bq\label{est_sp}
\int_{\R^d} \mu(\rho)\,dx  = \int_{\R^d} \rho^\delta\,dx \leq m_\rho^{\frac{\gamma-\delta}{\gamma-1}}(\gamma-1)^{\frac{\delta-1}{\gamma-1}}E_i(t)^{\frac{\delta-1}{\gamma-1}} \leq m_\rho^{\frac{\gamma-\delta}{\gamma-1}}(\gamma-1)^{\frac{\delta-1}{\gamma-1}}E_0^{\frac{\delta-1}{\gamma-1}}, 
\]%\eq
due to \eqref{est_mu} and Lemma \ref{lem_energy}. Thus we get
\bq\label{est_sp}
M_\mu \leq m_\rho^{\frac{\gamma-\delta}{\gamma-1}}(\gamma-1)^{\frac{\delta-1}{\gamma-1}}E_0^{\frac{\delta-1}{\gamma-1}}.
\eq
Moreover, it is obvious to get $M_\lambda = 2(\delta - 1)M_\mu >0$.

(i) For $1 < \gamma < 1 + \frac{2}{d}$ and $\delta = \gamma$, it follows from \eqref{est_j}, Lemmas \ref{lem_mww2}, \ref{lem_mw3}, and \ref{lem_gd} that
\[
\frac{C_0}{(I_0 + C_1 t + C_2 t^2)^{\frac{d(\gamma-1)}{2}}} \leq \frac{C_0}{I_\rho(t)^{\frac{d(\gamma-1)}{2}}} \leq E_i(t) \leq \frac{J_0 e^{(\gamma-1)\lt(1 + d(\delta-1) \rt)/2}}{(t+1)^{d(\gamma-1)}}.
\]
This yields
\[
\frac{C_0}{(I_0 + C_1 t + C_2 t^2)^{\frac{d(\gamma-1)}{2}}} \leq \frac{J_0 e^{(\gamma-1)\lt(1 + d(\delta-1) \rt)/2}}{(t+1)^{d(\gamma-1)}}.
\]
Using the same argument as in the proof of Theorem \ref{main_thm}, we conclude that the life-span $T$ of classical solutions with the initial data satisfying $C_0 > J_0 e^{(\gamma-1)\lt(1 + d(\delta-1) \rt)/2}C_2$ should be finite.

(ii) For $1 < \gamma < 1 + \frac{2}{d}$ and $\gamma > \delta > \gamma - \frac1d$, we use the similar argument as before to obtain
$$\begin{aligned}
&\frac{C_0}{(I_0 + C_1 t + C_2 t^2)^{\frac{d(\gamma-1)}{2}}} \cr
&\qquad \leq \lt(J_0^{\frac{\gamma-\delta}{\gamma-1}} +\frac{\lt(1 + d(\delta-1) \rt) \,m_\rho^{\frac{\gamma-\delta}{\gamma-1}}(\gamma-1)^{\frac{\delta-\gamma}{\gamma-1}}(\gamma-\delta)}{2(1 - d(\gamma - \delta))}\rt)^{\frac{\gamma-1}{\gamma-\delta}}\frac{1}{(t+1)^{d(\gamma-1)}}.
\end{aligned}$$
Hence if the initial data satisfy
\[
C_0 > \lt(J_0^{\frac{\gamma-\delta}{\gamma-1}} +\frac{\lt(1 + d(\delta-1) \rt) \,m_\rho^{\frac{\gamma-\delta}{\gamma-1}}(\gamma-1)^{\frac{\delta-\gamma}{\gamma-1}}(\gamma-\delta)}{2(1 - d(\gamma - \delta))}\rt)^{\frac{\gamma-1}{\gamma-\delta}}C_2,
\]
the life-span $T$ of classical solutions should be finite.
\end{proof}

%%%%%%%%%%%%%%%%%%%%%%%%%%%%%%%%%
%
%   \subsection{Proof of Corollary \ref{main_coro}}
%
%%%%%%%%%%%%%%%%%%%%%%%%%%%%%%%%%

\subsection{Proof of Corollary \ref{main_coro}} In this part, we provide the proof of Corollary \ref{main_coro} in which the viscosity coefficients $\mu$ and $\lambda$ do not depend on the fluid density $\rho$. Since the proof is very similar with that of Theorem \ref{main_thm}, we sketch the proof as follows. \newline

$\bullet$ {\it Estimate of the second-derivative of the total momentum of inertia}: We easily find
\begin{align}\label{est_ii}
\begin{aligned}
\frac{d^2}{dt^2}I(t) &=\int_{\R^d} \rho |u|^2\,dx + d\int_{\R^d} p(\rho)\,dx + \int_{\R^d \times \R^d} |v|^2 f\,dxdv \cr
&= 2E_k(t) + d(\gamma - 1)E_i(t) +2E_f(t).
\end{aligned}
\end{align}
Furthermore, we get
\[
2E_k(t) + d(\gamma - 1)E_i(t) +2E_f(t) \leq \max\{2, d(\gamma - 1)\}E(t) \leq \max\{2, d(\gamma - 1)\}E_0,
\]
due to Lemma \ref{lem_energy}. On the other hand, since $\gamma > 1$, we also obtain from \eqref{est_ee} that
\[
\frac{d^2}{dt^2}I(t) \geq 2E_k(t) + 2E_f(t) \geq \frac{M_0^2}{max\{m_\rho, m_f\}}.
\]
Thus we obtain
\[
\frac{M_0^2}{max\{m_\rho, m_f\}} \leq \frac{d^2}{dt^2}I(t) \leq \max\{2, d(\gamma - 1)\}E_0.
\]
In particular, we have
\bq\label{lem_mw2}
I(t) \leq I_0 + W_0 t + \max\lt\{1, \frac{d(\gamma - 1)}{2}\rt\}E_0 \,t^2.
\eq

$\bullet$ {\it Estimate of upper bound of $J$}: In a similar way as in Lemma \ref{lem_mww4}, we get
\bq\label{est_j2}
J_\rho(t) \geq (t+1)^2 E_i(t) \quad \mbox{and} \quad J_f(t) \geq 0.
\eq
Moreover, we also find from \eqref{est_ii} that 
$$\begin{aligned}
\frac{d}{dt}J(t) &= -(t+1)\frac{d}{dt}W(t) + 2(t+1)E(t) + (t+1)^2\frac{d}{dt}E(t)\cr
&\leq -(t+1)(2E_k(t) + d(\gamma - 1)E_i(t) +2E_f(t)) + 2(t+1)E(t)\cr
&= (2 - d(\gamma-1))(t+1)E_i(t).
\end{aligned}$$
Thus, if $2 - d(\gamma-1) > 0$, we get
\[
\frac{d}{dt}J(t) \leq \frac{2 - d(\gamma-1)}{(t+1)} J_\rho(t) \leq \frac{2 - d(\gamma-1)}{(t+1)}J(t),
\]
due to \eqref{est_j2}. On the other hand, if $2 - d(\gamma-1) \leq 0$, we obtain
\[
\frac{d}{dt}J(t) \leq 0.
\]
By solving the above differential inequalities, we have
\bq\label{est_jj2}
J \leq \left\{ \begin{array}{ll}
 J_0(t+1)^{2 - d(\gamma - 1)} & \textrm{if $2 - d(\gamma-1) > 0$,}\\[1mm]
 J_0 & \textrm{if $2 - d(\gamma-1) \leq 0$.}
  \end{array} \right.
\eq

$\bullet$ {\it Conclusion of the desired result:} For $2 - d(\gamma-1) \geq 0$, it follows from Lemma \ref{lem_mw3}, \eqref{est_j2}, and \eqref{est_jj2} that
\[
\frac{C_1}{I_\rho(t)^{\frac{d(\gamma-1)}{2}}} \leq E_i(t) \leq \frac{J_0}{(t+1)^{d(\gamma - 1)}}.
\]
On the other hand, we also find from \eqref{lem_mw2} that  
\[
I_\rho(t) \leq I(t) \leq I_0 + W_0 \, t + \max\{1, d(\gamma - 1)/2\} E_0 \, t^2
\]
This yields
\[
\frac{C_1}{(I_0 + W_0 \, t + \max\{1, d(\gamma - 1)/2\} E_0 \, t^2)^{\frac{d(\gamma-1)}{2}}} \leq E_i(t)\leq \frac{J_0}{(t+1)^{d(\gamma - 1)}}.
\]
Thus we obtain
\[
\frac{C_1}{(I_0 + W_0 \, t + \max\{1, d(\gamma - 1)/2\} E_0 \, t^2)^{\frac{d(\gamma-1)}{2}}} \leq \frac{J_0}{(t+1)^{d(\gamma - 1)}}.
\]
Hence we conclude that the life-span $T$ of classical solutions with the initial data satisfying $C_1 > (\max\{1, d(\gamma - 1)/2\} E_0)^{\frac{d(\gamma-1)}{2}} J_0$ should be finite.

%%%%%%%%%%%%%%%%%%%%%%%%%%%%%%%%%%%%%%%%%%%%%%%%%%%%%%%%%%%%%%%%%%%%%%%%%%%%%%%%%%%%%%%%%%%%%%%%%%%%%%%%%%%%%%%%%%%%%%%%%%%%%%%%%%%%%%%%%%%%%%
%
%
%         \section{Further extensions}
%
%
%%%%%%%%%%%%%%%%%%%%%%%%%%%%%%%%%%%%%%%%%%%%%%%%%%%%%%%%%%%%%%%%%%%%%%%%%%%%%%%%%%%%%%%%%%%%%%%%%%%%%%%%%%%%%%%%%%%%%%%%%%%%%%%%%%%%%%%%%%%%%%%%
\section{Further extensions}\label{sec_ext}
In this section, we want to show that analogous results to that of Theorems \ref{main_thm} and \ref{main_thm2} are also valid for other equations with suitable modifications. We begin by showing the adaptation of the strategy for the isentropic Navier-Stokes equations and two-phase fluid model consisting of the pressureless Euler equations and the isentropic compressible Navier-Stokes equations. Finally, we extend this strategy to thick sprays case in which the volume fraction of gas $\alpha$ is taken into account.
%%%%%%%%%%%%%%%%%%%%%%%%%%%%%%%%%%%%%%%%%%%%%%%%%%%%%%%%%%%%%%%%%%%%%%%%%%%%%%%%%%%%%%%%%%%%%%%%%%%%%%%%%%%%%%%%%%%%%%%%%%%%%%%%%%%%%%%%%%%%%%
%
%
%       \subsection{Compressible flows}
%
%
%%%%%%%%%%%%%%%%%%%%%%%%%%%%%%%%%%%%%%%%%%%%%%%%%%%%%%%%%%%%%%%%%%%%%%%%%%%%%%%%%%%%%%%%%%%%%%%%%%%%%%%%%%%%%%%%%%%%%%%%%%%%%%%%%%%%%%%%%%%%%%%%
\subsection{Isentropic compressible Navier-Stokes equations}\label{sec_icns} In this part, we are interested in the blow up analysis for the compressible fluids. By taking into account the case in which there is no particle interactions, i.e., $f \equiv 0$, we reduce the isentropic compressible Navier-Stokes equations from \eqref{main_eq}: 
\begin{align}\label{main_eq2}
\begin{aligned}
&\pa_t \rho + \nabla_x \cdot (\rho u) = 0, \qquad x \in \R^d, \quad t > 0,\cr
&\pa_t (\rho u) + \nabla_x \cdot (\rho u \otimes u) + \nabla_x p(\rho) -\nabla_x \cdot(2\mu(\rho) \mathbb{T}(u)) - \nabla_x (\lambda(\rho)\nabla_x \cdot u) = 0,
\end{aligned}
\end{align}
with the initial data
\bq\label{ini_main_eq2}
(\rho(x,0), u(x,0)) =: (\rho_0(x),u_0(x)), \quad x \in \R^d.
\eq
Let us remind the reader that the pressure law $p$ and the stress tensor $\mathbb{T}$ are given by
\bq\label{as_pt}
p(\rho) = \rho^\gamma \quad \mbox{with} \quad \gamma > 1 \quad \mbox{and} \quad \mathbb{T}(u) = \frac12\lt( \nabla_x u + (\nabla_x u)^T\rt),
\eq
and the Lam\'e viscosity coefficients satisfy
\bq\label{as_lm2}
\lambda(\rho) = 2\rho \mu^\prime (\rho) - 2\mu(\rho) \quad \mbox{and} \quad \mu(\rho) \geq 0.
\eq
%In particular, we deal with the case 
%\bq\label{as_vis}
%\mu(\rho) = \rho^\delta \quad \mbox{with} \quad \delta > 1, \quad \mbox{i.e.,} \quad \lambda(\rho) = 2(\delta - 1)\rho^\delta,
%\eq
%due to \eqref{as_lm2}.

Before we present our main result in this subsection, we define a solution space $\mathfrak{S}_1$ as follows.
\begin{definition}\label{def_isenNS} For any $T>0$, we call $(\rho,u) \in \mathfrak{S}_1(T)$ if $(\rho,u)$ is a classical solution to the Cauchy problem \eqref{main_eq2}-\eqref{ini_main_eq2} in the time interval $[0,T]$ satisfying the following conditions of decay at far fields:
\[
\mu(\rho)|u| \to 0, \quad \lt(p(\rho) + \lt( \mu(\rho) + \lambda(\rho)\rt)|\nabla_x u| \rt)\lt( |x| + |u|\rt) \to 0 \quad \mbox{as} \quad |x| \to \infty, \quad \forall \,t \in [0,T],
\]
and
\[
\rho|u|\lt(|x|^2 + |u|^2 + |u||\nabla_x u| \rt) \to 0\quad \mbox{as} \quad |x| \to \infty, \quad \forall \,t \in [0,T].
\]
\end{definition}
In a recent paper \cite{JWX}, the blow-up estimate of smooth solution to the system \eqref{main_eq2} with the viscosity coefficient $\mu(\rho) = \rho^\delta$ with $\delta > 1$ is provided. To be more precise, if $\gamma \in (1,1 + \frac{2}{d})$, $\frac{\gamma + 1}{2} < \delta \leq \gamma$, and the initial data $(\rho_0,u_0)$ satisfy suitable assumptions, then the life-span $T$ of classical solutions satisfying certain decay conditions to the system \eqref{main_eq2}-\eqref{ini_main_eq2} should be finite.  

Note that our strategies used for Theorem \ref{main_thm2} can be directly applied to the system \eqref{main_eq2} by setting $f \equiv 0$ in physical quantities appeared in Section \ref{sec_pre1}. More specifically, we have the analysis of finite time blow up of smooth solutions to the system \eqref{main_eq2} in the theorem below. Since the proof is almost same as before, we omit it here.

\begin{theorem}\label{main_thm3}Let $(\rho,u)$ be a solution to the Cauchy problem \eqref{main_eq2}-\eqref{ini_main_eq2} satisfying $(\rho,u) \in \mathfrak{S}_1(T)$.

$\diamond$ (General case) Suppose that $1 < \gamma < 1 + \frac{1}{d}$ and the viscosity coefficient $\mu$ and $\lambda$ satisfy \eqref{as_lm2}. Furthermore, the initial data \eqref{ini_main_eq3} satisfy
\[
C_0 > \frac12\lt(\frac{M_\mu}{2} + \max\{2,d(\gamma - 1)\}(E_k(0) + E_i(0))\rt)\lt(\tilde J_0 + \frac{M_\mu/2 + dM_\lambda/4}{1 - d(\gamma-1)}\rt),
\]
Here $C_0$ and $M_\mu$ are given as in Theorem \ref{main_thm}, and $\tilde{J_0}$ is given by 
\[
\tilde{J_0} := I_\rho(0) - W_\rho(0) + (E_k(0) + E_i(0)) \geq 0.
\]

$\diamond$ (Particular case) Suppose that the viscosity coefficient $\mu$ has a form of $\mu(\rho) = \rho^\delta$ with $\delta \in (1,\gamma]$ and $1 < \gamma < 1 + \frac{2}{d}$.
\begin{itemize} 
\item If $\delta = \gamma$ and the initial data satisfy
\[
C_0 > \frac{\tilde{J_0} e^{(\gamma-1)\lt(1 + d(\delta-1) \rt)/2}}{2}\lt(\frac{M_\mu}{2} + \max\{2,d(\gamma - 1)\}(E_k(0) + E_i(0))\rt).
\]
\item If $\gamma - \frac1d< \delta < \gamma$ and the initial data satisfy
\[
C_0 >\frac{C_4}{2}\lt(\frac{M_\mu}{2} + \max\{2,d(\gamma - 1)\}(E_k(0) + E_i(0))\rt),
\]
where $C_4$ is a positive constant given by
\[
C_4 := \lt(\tilde{J_0}^{\frac{\gamma-\delta}{\gamma-1}} +\frac{\lt(1 + d(\delta-1) \rt) \,m_\rho^{\frac{\gamma-\delta}{\gamma-1}}(\gamma-1)^{\frac{\delta-\gamma}{\gamma-1}}(\gamma-\delta)}{2(1 - d(\gamma - \delta))}\rt)^{\frac{\gamma-1}{\gamma-\delta}}.
\]
\end{itemize}
Then the life-span $T$ of the solution $(\rho,u)$ is finite.
\end{theorem}

We emphasize that Theorem \ref{main_thm3} extends the blow-up estimates in \cite[Theorems 2.4 and 2.5]{JWX} since $\gamma - \frac1d < \frac{\gamma + 1}{2}$ if and only if $\gamma < 1 + \frac2d$, i.e., our regime for $\delta$ is always larger then that of \cite[Theorems 2.4 and 2.5]{JWX}. Furthermore, Theorem \ref{main_thm3} deals with the case $\mu(\rho) = \rho$, i.e., $\delta = 1$ which is covered in \cite{JWX}. 
%%%%%%%%%%%%%%%%%%%%%%%%%%%%%%%%%%%%%%%%%%%%%%%%%%%%%%%%%%%%%%%%%%%%%%%%%%%%%%%%%%%%%%%%%%%%%%%%%%%%%%%%%%%%%%%%%%%%%%%%%%%%%%%%%%%%%%%%%%%%%%
%
%
%        \subsection{Two-phase models}
%
%
%%%%%%%%%%%%%%%%%%%%%%%%%%%%%%%%%%%%%%%%%%%%%%%%%%%%%%%%%%%%%%%%%%%%%%%%%%%%%%%%%%%%%%%%%%%%%%%%%%%%%%%%%%%%%%%%%%%%%%%%%%%%%%%%%%%%%%%%%%%%%%%%
\subsection{Two-phase fluid equations: Pressureless Euler/Navier-Stokes equations}\label{sec_tpf} In this part, we consider the coupled hydrodynamic system consisting of the pressureless Euler equations and the barotropic Navier-Stokes equations where the coupling is through the drag force. More precisely, this system is governed by
\begin{align}\label{main_eq3}
\begin{aligned}
&\pa_t n + \nabla_x \cdot(n w) = 0, \qquad x \in \R^d, \quad t > 0,\cr
&\pa_t (nw) + \nabla_x \cdot(n w \otimes w) = D(\rho,u-w)n,\cr
&\pa_t \rho + \nabla_x \cdot (\rho u) = 0,\cr
&\pa_t (\rho u) + \nabla_x \cdot (\rho u \otimes u) + \nabla_x p(\rho) -\nabla_x \cdot\lt(2\mu(\rho) \mathbb{T}(u)) + \lambda(\rho)(\nabla_x \cdot u)\mathbb{I}\rt) = -D(\rho,u-w)n,
\end{aligned}
\end{align}
with the initial data
\bq\label{ini_main_eq3}
(n(x,0), w(x,0), \rho(x,0), u(x,0)) =: (n_0(x), w_0(x), \rho_0(x),u_0(x)), \quad x \in \R^d.
\eq
Here the pressure law $p$ and the stress tensor $\mathbb{T}$ given by \eqref{as_pt}, and the drag force $D$ and the viscosity coefficients $\lambda$ and $\mu$ satisfy \eqref{as_lm2}

We first give a brief outline for the formal derivation of the system \eqref{main_eq2} from \eqref{main_eq} without the particle interaction operator $Q(f,f)$. For this, we introduce the macroscopic variables of the local mass $n$ and momentum $nw$ for the distribution function $f$ as follows.
\[
n(x,t) := \int_{\R^d} f(x,v,t)\,dv \quad \mbox{and} \quad (nw)(x,t):=\int_{\R^d} v f(x,v,t)\,dv \quad \mbox{for} \quad (x,t) \in \R^d \times \R_+.
\]
Integrating the equation $\eqref{main_eq}_1$ in $v$, we obtain  the continuity equation:
\[
\pa_t n + \nabla_x \cdot(n w) = 0.
\]
For the momentum equation, we multiply $\eqref{main_eq}_1$ by $v$ and integrate it with respect to $v$ to find that
\[
\pa_t(nw)  + \nabla_x \cdot (n w \otimes w) + \nabla_x \cdot \hat\sigma = \int_{\R^d} D(\rho,u-v)f\,dv = D(\rho,u-v)n,
\]
where $\hat\sigma$ denotes the pressure tensor which is given by 
\[
\hat\sigma(x,t) := \int_{\R^d} (v - w(x,t))\otimes (v - w(x,t)) f(x,v,t)\, dv.
\]
Then, we multiply the equations $\eqref{main_eq}_1$ by $|v|^2/2$ and integrate in $v$ to deduce that
\begin{align}\label{f-ener-func1}
\begin{aligned}
\frac12\frac{d}{dt}\int_{\R^d} |v|^2 f\,dv &= -\frac12\int_{\R^d} |v|^2 \lt( \nabla_x \cdot (v f) + \nabla_v \cdot (D(\rho,u - v)f) \rt) dv =:I_1 + I_2,
\end{aligned}
\end{align}
where $I_1$ and $I_2$ are estimated as
\begin{align}\label{f-ener-func2}
\begin{aligned}
I_1 &= - \frac12\nabla_x \cdot \lt( \int_{\R^d} |v - w|^2 v f \,dv + n|w|^2 w + 2\int_{\R^d} (v - w)\otimes (v - w) w f\, dv\rt)\cr
&= - \nabla_x \cdot \lt( q + n(|w|^2 + \theta)w + \hat\sigma w\rt),\cr
I_2 &= \intr v \cdot D(\rho,u-v)f\,dv.
\end{aligned}
\end{align}
Here $q$ and $\theta$ represent the energy-flux and the temperature, respectively:
$$\begin{aligned}
q(x,t) &:= \frac12\int_{\R^d} |v - w(x,t)|^2 (v - w(x,t)) f(x,v,t)\, dv, \cr
(n\theta)(x,t)&:= \frac12\int_{\R^d} |v - w(x,t)|^2 f(x,v,t)\,dv.
\end{aligned}$$
Combining \eqref{f-ener-func1} and \eqref{f-ener-func2}, we obtain
\[
\pa_t\lt( n\lt(\frac{|w|^2}{2} + \theta\rt)\rt) + \nabla_x \cdot \lt( \lt(n(|w|^2 + \theta) + \hat\sigma\rt)w + q\rt) = \intr v \cdot D(\rho,u-v)f\,dv.
\]
Hence we collect all the equations of macroscopic variables and those of the compressible fluid variables $(\rho,u)$ to have
\begin{align}\label{comp-sys}
\begin{aligned}
&\pa_t n + \nabla_x \cdot(nw) = 0,\quad (x,t) \in \R^d \times \R_+,\cr
&\pa_t(nw) + \nabla_x \cdot (nw \otimes w) + \nabla_x \cdot \hat\sigma = D(\rho,u-v)n,\cr
&\pa_t\lt( n\lt(\frac{|w|^2}{2} + \theta\rt)\rt) + \nabla_x \cdot \lt( \lt(n(|w|^2 + \theta) + \hat\sigma\rt)w + q\rt)= \intr v \cdot D(\rho,u-v)f\,dv,\cr
&\pa_t \rho + \nabla_x \cdot (\rho u) = 0,\cr
&\pa_t (\rho u) + \nabla_x \cdot (\rho u \otimes u) + \nabla_x p(\rho) -\nabla_x \cdot(2\mu(\rho) \mathbb{T}(u)) - \nabla_x (\lambda(\rho)\nabla_x \cdot u) = -D(\rho,u-v)n.
\end{aligned}
\end{align}
Note that the system \eqref{comp-sys} is not closed due to the energy-flux $q$ and the general form of drag force $D$. In order to close the system \eqref{comp-sys}, we assume that the velocity distribution is mono-kinetic, i.e., $f(x,v,t) = n(x,t)\delta(v - w(x,t))$, where $\delta$ denotes the standard Dirac delta function. Then the system \eqref{comp-sys} reduces to the two-phase fluid equations \eqref{main_eq3}. It is worth noticing that the system \eqref{main_eq3} can also be formally derived from \eqref{main_eq} with the strong inelastic collision effect between particles employing the similar argument as in \cite{DM}. For the hydrodynamic limit from particle-fluid equations to the two-phase fluid system, we refer the reader to \cite{CCK,CG,Ch,GJV1,GJV2,MV3} and references therein.

For the blow-up analysis, we introduce the physical quantities, mass, momentum, momentum weight, momentum of inertia, and energy similarly as in Section \ref{sec_pre1}:
$$\begin{aligned}
&n_c(t) := \intr n\,dx, \quad M_n(t):= \intr n w \,dx, \quad W_n(t):= \intr n w \cdot x\,dx,\cr
&I_n(t):= \frac12 \intr n|x|^2\,dx, \quad \mbox{and} \quad E_n(t):= \frac12 \intr n|w|^2\,dx.
\end{aligned}$$

Using these physical quantities, we first provide basic estimates, mass, total momentum, and total energy as in the lemma below. Since the idea of proof is similar to that of Lemma \ref{lem_energy}, we skip it here.
\begin{lemma}Let $(n,w,\rho,u)$ be a classical solution to the system \eqref{main_eq3}-\eqref{ini_main_eq3} in the interval $[0,T]$. Then we have
$$\begin{aligned}
&(i) \,\,\,\, \frac{d}{dt}n_c(t) = \frac{d}{dt}\lt(M_\rho(t) + M_n(t) \rt) = 0,\cr
&(ii) \,\, \frac{d}{dt}\lt(E_k(t) + E_i(t) + E_n(t) \rt) + \int_{\R^d} \lt(2\mu(\rho) \mathbb{T}(u) : \mathbb{T}(u)+ \lambda(\rho)|\nabla \cdot u|^2\rt)dx\cr
&\qquad \qquad+ \intr D(\rho,u-w)n \cdot (u-w) \,dx = 0, 
\end{aligned}$$
for all $t \in [0,T]$, where $M_\rho$, $E_k$, and $E_i$ are given as in Section \ref{sec_pre1}.
\end{lemma}
We also define a solution space $\mathfrak{S}_2$ for the system \eqref{main_eq3} as follows.
\begin{definition} For any $T>0$, we call $(n,w,\rho,u) \in \mathfrak{S}_2(T)$ if $(n,w,\rho,u)$ is a classical solution to the Cauchy problem \eqref{main_eq3}-\eqref{ini_main_eq3} in the time interval $[0,T]$ satisfying the following conditions of decay at far fields:
\[
n|w|\lt(|x|^2 + |w|^2 + |w||\nabla_x w| \rt) \to 0\quad \mbox{as} \quad |x| \to \infty, \quad \forall \,t \in [0,T],
\]
\[
\mu(\rho)|u| \to 0, \quad \lt(p(\rho) + \lt( \mu(\rho) + \lambda(\rho)\rt)|\nabla_x u| \rt)\lt( |x| + |u|\rt) \to 0 \quad \mbox{as} \quad |x| \to \infty, \quad \forall \,t \in [0,T],
\]
and
\[
\rho|u|\lt(|x|^2 + |u|^2 + |u||\nabla_x u| \rt) \to 0\quad \mbox{as} \quad |x| \to \infty, \quad \forall \,t \in [0,T].
\]
\end{definition}
Again, the proof for the estimate of blow-up of classical solutions in finite time is very similar with that of Theorems \ref{main_thm} and \ref{main_thm2}. Thus we state the main theorem in this part without providing its proof.
\begin{theorem}\label{main_thm4}Let $(n,w,\rho,u)$ be a solution to the Cauchy problem \eqref{main_eq3}-\eqref{ini_main_eq3} satisfying $(n,w,\rho,u) \in\mathfrak{S}_2(T)$.

$\diamond$ (General case) Suppose that $1 < \gamma < 1 + \frac{1}{d}$ and the viscosity coefficient $\mu$ and $\lambda$ satisfy \eqref{as_lm2}. Furthermore, the initial data \eqref{ini_main_eq3} satisfy
\[
C_0 > \frac12\lt(\frac{M_\mu}{2} + \max\{2,d(\gamma - 1)\}E^1_0\rt) \lt(J_0^1 + \frac{M_\mu/2 + dM_\lambda/4}{1 - d(\gamma-1)}\rt),
\]
where $E^1_0 := E_k(0) + E_i(0) + E_n(0)$,  $J^1_0:= I_\rho(0) + I_n(0) - W_\rho(0) - W_n(0) + E^1_0$, and $C_0$ and $M_\mu$ are given as in Theorem \ref{main_thm}, i.e., 
\[
C_0 := \lt(\frac{\pi^{d/2}}{\Gamma\lt(d/2 +1 \rt)}\rt)^{1 - \gamma}\frac{m_\rho^{\frac{(d+2)\gamma - d}{2}}}{2^{\frac{(d+2)\gamma - d}{2}}(\gamma-1)}, \quad \|\mu(\rho)\|_{L^\infty(0,T;L^1(\R^d))} \leq M_\mu,
\]
Then the life-span $T$ of the solution $(n,w,\rho,u)$ is finite.

$\diamond$ (Particular case) Suppose that the viscosity coefficient $\mu$ has a form of $\mu(\rho) = \rho^\delta$ with $\delta \in (1,\gamma]$.
\begin{itemize}
\item If $1 < \gamma < 1 + \frac{2}{d}$, $\delta = \gamma$, and the initial data satisfy
\[
C_0 > \frac{J^1_0 e^{(\gamma-1)\lt(1 + d(\delta-1) \rt)/2}}{2}\lt(\frac{M_\mu}{2} + \max\{2,d(\gamma - 1)\}E^1_0\rt).
\]
\item If $1 < \gamma < 1 + \frac{2}{d}$, $\gamma - \frac1d< \delta < \gamma$ and the initial data satisfy
\[
C_0 >\frac{C_5}{2}\lt(\frac{M_\mu}{2} + \max\{2,d(\gamma - 1)\}E^1_0\rt)\lt((J^1_0)^{\frac{\gamma-\delta}{\gamma-1}} +\frac{\nu \,m_\rho^{\frac{\gamma-\delta}{\gamma-1}}(\gamma-1)^{\frac{\delta-\gamma}{\gamma-1}}(\gamma-\delta)}{1 - d(\gamma - \delta)}\rt)^{\frac{\gamma-1}{\gamma-\delta}},
\]
where $C_5$ is a positive constant given by
\[
C_5 := \lt((J^1_0)^{\frac{\gamma-\delta}{\gamma-1}} +\frac{\lt(1 + d(\delta-1) \rt) \,m_\rho^{\frac{\gamma-\delta}{\gamma-1}}(\gamma-1)^{\frac{\delta-\gamma}{\gamma-1}}(\gamma-\delta)}{2(1 - d(\gamma - \delta))}\rt)^{\frac{\gamma-1}{\gamma-\delta}}.
\]
\end{itemize}
Then the life-span $T$ is finite.
\end{theorem}

\begin{remark}In \cite{CK2}, the system \eqref{main_eq3} with constant viscosity coefficients and density independent drag forcing term $D = u-w$ in the three dimensional periodic domain is studied. By using a {\it a priori} estimate of large-time behavior together with the bootstrapping argument, global classical solutions in time are constructed. 
\end{remark}
%%%%%%%%%%%%%%%%%%%%%%%%%%%%%%%%%%%%%%%%%%%%%%%%%%%%%%%%%%%%%%%%%%%%%%%%%%%%%%%%%%%%%%%%%%%%%%%%%%%%%%%%%%%%%%%%%%%%%%%%%%%%%%%%%%%%%%%%%%%%%%
%
%
%         subSection: Beyond the moderately thick sprays
%
%
%%%%%%%%%%%%%%%%%%%%%%%%%%%%%%%%%%%%%%%%%%%%%%%%%%%%%%%%%%%%%%%%%%%%%%%%%%%%%%%%%%%%%%%%%%%%%%%%%%%%%%%%%%%%%%%%%%%%%%%%%%%%%%%%%%%%%%%%%%%%%%%%
\subsection{Thick sprays model}
In this part, we present the modelling for thick sprays where the volume fraction of particle is not negligible(see \cite{Rou} for the concept of thick sprays). We show that the fluid still dominate the system, i.e., the kinetic equation does not have an effect to prevent the formation of the finite time blow up of classical solutions even if we take into account more thicker particles immersed in the fluid than the previous case, moderately thick sprays. 

Consider the thick sprays which can be governed by
\begin{align}\label{thick}
\begin{aligned}
&\pa_t f + v \cdot \nabla_x f + \nabla_v \cdot (\Gamma f) = Q(f,f), \quad (x,v,t) \in {\R^d} \times \R^d \times \R_+,\cr
&\pa_t (\alpha\rho) + \nabla_x \cdot (\alpha\rho u) = 0,\cr
&\pa_t (\alpha \rho u) + \nabla_x \cdot (\alpha \rho u \otimes u) + \nabla_x p(\rho) = -\int_{\R^d} m_p \Gamma f\,dv,\cr
&1 - \alpha = \frac{m_p}{\rho_p}\int_{\R^d}  f\,dv,
\end{aligned}
\end{align}
with the initial data
\bq\label{ini_thick}
(f(x,v,0), \rho(x,0), u(x,0), \alpha(x,0)) =: (f_0(x,v),\rho_0(x),u_0(x), \alpha_0(x)), \quad (x,v) \in \R^d \times \R^d,
\eq
where $\alpha := \alpha(x,t) \in [0,1]$ denotes the volume fraction of gas at time $t \in \R_+$ and point $x \in {\R^d}$. Here 
\[
m_p \Gamma = -\frac{m_p}{\rho_p} \nabla_x p + D(\rho,u-v) \quad \mbox{with} \quad m_p = |B(0,r)| \rho_p\mbox{ ($m_p, \rho_p, r$ are positive constants)},
\]
the drag force $D$ and the particle interaction operator $Q$ satisfy {\bf (H1)} and {\bf (H2)}, respectively. Note that the momentum equation $\eqref{thick}_3$ can be rewritten as 
\[
\alpha \rho\pa_t u + \alpha \rho u \cdot \nabla_x u + \alpha \nabla_x p = - \intr D(\rho,u-v) f\,dv,
\]
due to $\eqref{thick}_2$ and $\eqref{thick}_4$. 

For the system \eqref{thick} with internal energy equations, global conservations, such as total energy, momentum, and mass are provided in \cite{BDM}, and the formal derivation of multiphase flows models including the volume fraction from \eqref{thick} with the inelastic collision operator is studied in \cite{DM}. We also refer to \cite{BDGN} for numerical study for the system \eqref{thick}.

Since we have one more equation for the volume fraction of the gas $\alpha$, we reintroduce the following $\alpha$-dependent physical quantities, mass, momentum, momentum weight, momentum of inertia, and energy as follows:
$$\begin{aligned}
& m_\rho^\alpha (t) := \int_{\R^d} \alpha \rho\,dx, \quad M_\rho^\alpha(t) := \into \alpha\rho u\,dx, \quad W_\rho^\alpha(t):= \into \alpha\rho u \cdot x\,dx,\cr
&I_\rho^\alpha(t)  := \frac12 \into \alpha\rho|x|^2\,dx, \quad E_k^\alpha(t):=\frac12 \into \alpha\rho|u|^2\,dx, \quad \mbox{and} \quad  E_i^\alpha(t):= \frac{1}{\gamma - 1}\into \alpha \rho^\gamma\,dx.
\end{aligned}$$

In a similar fashion as in the previous sections, we define our notation of classical solutions as follows.
\begin{definition} For any $T>0$, we call $(f,\rho,u,\alpha) \in \mathfrak{S}_3(T)$ if $(n,w,\rho,u)$ is a classical solution to the Cauchy problem \eqref{thick}-\eqref{ini_thick} in the time interval $[0,T]$ satisfying the following conditions of decay at far fields:
\[
f|v|\lt( |x|^2 + |v|^2\rt) \to 0 \quad \mbox{as} \quad |x| \to \infty, \quad \forall \,(v,t) \in \R^d \times [0,T],
\]
\[
f\lt(|\nabla_x p(\rho)|+ |D(\rho,u-v)|\rt)\lt( |x|^2 + |v|^2\rt) \to 0 \quad \mbox{as} \quad |v| \to \infty, \quad \forall \,(x,t) \in \R^d \times [0,T],
\]
and
\[
\alpha p(\rho) \lt( |x| + |u|\rt) \to 0, \quad \alpha\rho|u|\lt(|x|^2 + |u|^2 + |u||\nabla_x u| \rt) \to 0 \quad \mbox{as} \quad |x| \to \infty, \quad \forall \,t \in [0,T].
\]
\end{definition}

Similarly, we first provide the energy estimates in the lemma below.

\begin{lemma}Let $(f,\rho,u,\alpha)$ be a classical solution to the system  \eqref{thick}-\eqref{ini_thick} in the interval $[0,T]$. Then we have
$$\begin{aligned}
&(i) \quad \frac{d}{dt} m_\rho^\alpha(t) = \frac{d}{dt} M^\alpha (t)= 0,\cr
&(ii) \,\,\,\, \frac{d}{dt}E^\alpha(t) + \int_{\R^d \times \R^d} D(\rho,u-v) \cdot (u-v) f\,dxdv = 0,
\end{aligned}$$
for all $t \in [0,T]$. Here $M^\alpha := M^\alpha_\rho + M_f$ and $E^\alpha := E^\alpha_k + E^\alpha_i + E_f$.
\end{lemma}
\begin{proof} The proof of $(i)$ is easily obtained. For the estimate of $(ii)$, we multiply the continuity equation $\eqref{thick}_2$ by $\rho^{\gamma-1}$ to find that $\alpha p(\rho)$ satisfies
\[
\pa_t (\alpha p(\rho)) + \nabla_x\cdot(\alpha p(\rho) u ) = \lt(1 - \frac{1}{\gamma} \rt)\lt( \alpha \pa_t p(\rho) + \alpha \nabla_x p(\rho) \cdot u\rt).
\] 
This implies 
\[
\frac{\gamma}{\gamma-1}\frac{d}{dt} \into \alpha p(\rho)\,dx = \into \alpha \pa_t p(\rho) + \alpha \nabla_x p(\rho)\cdot u\,dx.
\]
Then we use the above equality together with the similar argument as in Lemma \ref{lem_energy} to obtain
$$\begin{aligned}
&\frac12 \frac{d}{dt}\lt(\into \alpha \rho |u|^2\,dx + \intor |v|^2f \,dxdv\rt) + \int_{\R^d \times \R^d} D(\rho,u-v) \cdot (u-v) f\,dxdv\cr
&\quad \quad =\intor \nabla_x p(\rho) \cdot u\,f\,dxdv - \into u \cdot \nabla_x p(\rho)\,dx - \intor \nabla_x p(\rho) \cdot v \,f\,dxdv\cr
&\quad \quad = \into p(\rho) \pa_t \alpha\,dx - \into \alpha \nabla_x p(\rho) \cdot u\,dx\cr
&\quad \quad = -\frac{1}{\gamma-1}\into \alpha p(\rho)\,dx,
\end{aligned}$$
due to
\[
1 - \alpha = \frac{m_p}{\rho_p}\intr f\,dv \quad \mbox{and} \quad \pa_t \alpha = -\pa_t \lt(\frac{m_p}{\rho_p}\intr f\,dv\rt) = \nabla_x \cdot \lt(\frac{m_p}{\rho_p}\intr vf\,dv\rt).
\]
This completes the proof.
\end{proof}
We set the total momentum weight $W^\alpha$ and momentum of inertia $I^\alpha$:
\[
W^\alpha := W^\alpha_\rho + W_f \quad \mbox{and} \quad I^\alpha := I^\alpha_\rho + I_f,
\]
respectively. In a similar way as in Lemma \ref{lem_conserv}, we estimate the total momentum of inertia and kinetic energies as follows.
\begin{lemma}It holds 
\[
\frac{d}{dt}I^\alpha_\rho(t) = W_\rho(t), \quad \frac{d}{dt}I_f(t) = W_f(t), \quad \mbox{i.e.,} \quad \frac{d}{dt} I^\alpha(t) = W^\alpha(t),
\]
and
\[
(M^\alpha(0))^2 \leq 2\max\{m^\alpha_\rho, m_f\}\lt(E^\alpha_k(t) + E_f(t)\rt).
\]
\end{lemma}
\begin{lemma}\label{lem_th0}Suppose $\rho \in L^\infty(\R^d \times [0,T])$. Then we have
\[
\frac{(M^\alpha_0)^2}{ \max\{ m_\rho^\alpha, m_f \}} \leq \frac{d^2}{dt^2} I^\alpha(t) \leq \max\{2, d(\gamma -1)\}E^\alpha_0 + d\|p\|_{L^\infty}m_f.
\]
\end{lemma}
\begin{proof}Using the similar argument as in the proof of Corollary \ref{main_coro}, we estimate the second time-derivative of the total momentum of inertia $I^\alpha$ as
$$\begin{aligned}
\frac{d^2}{dt^2} I^\alpha(t) &= \frac{d}{dt}W^\alpha(t) \cr
&= \intr \alpha \rho|u|^2\,dx + d\intr p(\rho)\,dx + \intrr |v|^2 f\,dxdv\cr
&= \intr \alpha \rho|u|^2\,dx + d\intr \alpha p(\rho)\,dx + \intrr |v|^2 f\,dxdv+ d\intr (1 - \alpha)p(\rho)\,dx\cr
&=2E^\alpha_k(t) + d(\gamma - 1)E^\alpha_i(t) + 2E_f(t) + \frac{d m_p}{\rho_p}\intrr p(\rho)\,f \,dxdv\cr
&\leq \max\{2, d(\gamma -1)\}E^\alpha(t) + \frac{d m_p}{\rho_p}\|p\|_{L^\infty}m_f\cr
&\leq \max\{2, d(\gamma -1)\}E^\alpha(0) + \frac{d m_p}{\rho_p}\|p\|_{L^\infty}m_f.
\end{aligned}$$
For the lower bound, we find
\[
\frac{d^2}{dt^2} I^\alpha(t) \geq \intr \alpha \rho|u|^2\,dx + \intrr |v|^2 f\,dxdv = 2(E^\alpha_k(t) + E_f(t)) \geq \frac{(M^\alpha(0))^2}{ \max\{ m_\rho^\alpha, m_f\}},
\]
due to $\alpha \rho \geq 0$.
\end{proof}
\begin{lemma}\label{lem_th1}There exists a positive constant $C_0^\alpha$ such that
\[
E^\alpha_i(t) \geq \frac{C^\alpha_0}{I^\alpha_\rho(t)^{\frac{d(\gamma-1)}{2}}}.
\]
Here $C^\alpha_0 > 0$ is explicitly given by
\[
C^\alpha_0 = \lt( \frac{\Gamma\lt(d/2 +1 \rt)}{\pi^{d/2}}\rt)^{\gamma-1}\frac{(m^\alpha_\rho)^{\frac{(d+2)\gamma - d}{2}}}{2^{\frac{(d+2)\gamma - d}{2}}(\gamma-1)}.
\]
\end{lemma}
\begin{proof}Since $\alpha \in [0,1]$, the proof can be obtained by using the similar argument as before.
\end{proof}
Similarly, we also provide the upper bound of $E^\alpha_i$ using the following function.
\[
J^\alpha := I^\alpha - (t+1)W^\alpha + (t+1)^2 E^\alpha.
\]
\begin{lemma}\label{lem_th2}According to the $\gamma > 1$, we have the following upper bounds of $J$. 
\begin{displaymath}
(t+1)^2 E_i^\alpha \leq J^\alpha \leq \left\{ \begin{array}{ll}
 J^\alpha_0(t+1)^{2 - d(\gamma - 1)} & \textrm{if $2 - d(\gamma-1) > 0$,}\\
 J^\alpha_0 & \textrm{if $2 - d(\gamma-1) \leq 0$.}
  \end{array} \right.
\end{displaymath}
\end{lemma}
\begin{proof}We set
\[
J^\alpha_\rho := I^\alpha_\rho - (t+1)W^\alpha_\rho + (t+1)^2 E^\alpha_\rho \quad \mbox{and} \quad J_f := I_f - (t+1)W_f + (t+1)^2 E_f.
\]
Using the similar argument as before, we find
\[
J^\alpha_\rho \geq (t+1)^2E^\alpha_i \quad \mbox{and} \quad J_f \geq 0.
\]
We also notice that
$$\begin{aligned}
\frac{d}{dt}J^\alpha &= -(t+1)\frac{d}{dt} W^\alpha + 2(t+1)E^\alpha + (t+1)^2\frac{d}{dt} E^\alpha\cr
&\leq -(t+1)\lt(2E^\alpha_k + d(\gamma-1)E^\alpha_i + 2E_f + \frac{d m_p}{\rho_p}\intrr p(\rho)\,f\,dxdv \rt) \cr
&\quad + 2(t+1)\lt( E^\alpha_k + E^\alpha_i + E_f\rt)\cr
&= (2 - d(\gamma-1))(t+1)E^\alpha_i - \frac{dm_p (t+1)}{\rho_p}\intrr p(\rho)\,f\,dxdv\cr
&\leq (2 - d(\gamma-1))(t+1)E^\alpha_i.
\end{aligned}$$
Hence if $2 - d(\gamma-1) > 0$, we get
\[
\frac{d}{dt}J^\alpha(t) \leq \frac{2 - d(\gamma-1)}{(t+1)} J^\alpha_\rho(t) \leq \frac{2 - d(\gamma-1)}{(t+1)}J^\alpha(t).
\]
On the other hand, if $2 - d(\gamma-1) \leq 0$, we obtain
\[
\frac{d}{dt}J^\alpha(t) \leq 0.
\]
\end{proof}
\begin{theorem}\label{thm_thick1}Let $(f,\rho,u,\alpha)$ be a solution to the Cauchy problem \eqref{thick}-\eqref{ini_thick} satisfying $(f,\rho,u,\alpha) \in \mathfrak{S}_3(T)$. If $\rho \in L^\infty(\R^d \times [0,T])$, $1 < \gamma \leq 1 + \frac{2}{d}$, and 
\bq\label{condi_thick}
C_0^\alpha > (\max\{1, d(\gamma - 1)/2\} E_0^\alpha + d\|p\|_{L^\infty}m_f/2)^{\frac{d(\gamma-1)}{2}} J^\alpha_0,
\eq
then the life-span $T$ of the solution $(f,\rho,u,\alpha)$ is finite.
\end{theorem}
\begin{proof}It follows from Lemmas \ref{lem_th1} and \ref{lem_th2} that
\[
\frac{C^\alpha_0}{I^\alpha_\rho(t)^{\frac{d(\gamma-1)}{2}}} \leq E_i^\alpha(t) \leq \frac{J_0^\alpha}{(t+1)^{d(\gamma-1)}}.
\]
Using the similar argument as before together with Lemma \ref{lem_th0}, we obtain
\[
\frac{C^\alpha_0}{\lt(I_0^\alpha + W_0^\alpha t + \lt( \max\{1, d(\gamma -1)/2\}E^\alpha_0 + d\|p\|_{L^\infty}m_f/2\rt)t^2\rt)^{\frac{d(\gamma-1)}{2}}} \leq \frac{J_0^\alpha}{(t+1)^{d(\gamma-1)}}.
\]
Hence we conclude that the life-span $T$ of classical solutions to the system \eqref{thick}-\eqref{ini_thick} with the initial data satisfying \eqref{condi_thick} should be finite.

\end{proof}

\begin{remark} Theorem \ref{thm_thick1} still holds for the system \eqref{thick} with the viscosity with a positive constant coefficient $C\Delta u$.
\end{remark}
%%%%%%%%%%%%%%%%%%%%%%%%%%%%%%%%%%%%%%%%%%%%%%%%%%%%%%%%%%%%%%%%%%%%%%%%%%%%%%%%%%%%%%%%%%%%%%%%%%%%%%%%%%%%%%%%%%%%%%%%%%%%%%%%%%%%%%%%%%%%%%
%
%
%         \appendix
%
%
%%%%%%%%%%%%%%%%%%%%%%%%%%%%%%%%%%%%%%%%%%%%%%%%%%%%%%%%%%%%%%%%%%%%%%%%%%%%%%%%%%%%%%%%%%%%%%%%%%%%%%%%%%%%%%%%%%%%%%%%%%%%%%%%%%%%%%%%%%%%%%%%

\appendix

\section{A Gronwall-type inequality}
In this Appendix, we provide a Gronwall-type inequality.  
\begin{lemma}\label{lem_usef2} Let us consider the following differential inequality:
\bq\label{est_diff}
f^\prime(t) \leq \frac{a}{t+1} f(t) + \frac{b}{(t+1)^{2\beta}}f^\beta,
\eq
where $a,b,\beta$ are positive constants, and $f \geq 0$. Suppose $\beta \leq 1$. 
\begin{itemize}
\item If $\beta = 1$, $f$ satisfies
\[
f(t) \leq f(0)e^b (t+1)^a.
\]
\item If $\beta < 1$ and $2\beta + a(1-\beta) = 1$, $f$ satisfies
\[
f(t)^{1-\beta} \leq f(0)^{1-\beta}(t+1)^{1-2\beta} + (1-2\beta)(t+1)^{1-2\beta}\ln(t+1).
\]
\item If $\beta < 1$ and $2\beta + a(1-\beta) \neq 1$, $f$ satisfies
\[
f(t)^{1-\beta} \leq f(0)^{1-\beta}(t+1)^{a(1-\beta)} + \frac{b(1-\beta)}{1 - (2\beta + a(1-\beta))}\lt( (t+1)^{1-2\beta} - (t+1)^{a(1-\beta)} \rt).
\]

\end{itemize}
\end{lemma}
\begin{proof} We divide the proof into two cases: $\beta = 1$ and $\beta < 1$. \newline

\noindent (i) Case $\beta = 1$: In this case, the differential inequality \eqref{est_diff} becomes
\[
f^\prime(t) \leq f^\prime(t) \leq \frac{a}{t+1} f(t) + \frac{b}{(t+1)^{2}}f(t).
\]
Using the integrating factor, we find
\[
\lt(\frac{f(t)}{(t+1)^a} \rt)^\prime \leq \frac{b}{(t+1)^{2 + a}}f(t).
\]
Set $g := f/(t+1)^a$, then $g$ satisfies
\[
g^\prime(t) \leq \frac{b}{(t+1)^2}g(t) \quad \mbox{with} \quad g(0) = f(0).
\]
This yields
\[
\lt(g(t)e^{b\lt(\frac{1}{t+1} - 1 \rt)} \rt)^\prime \leq 0, \quad \mbox{i.e.,} \quad g(t) \leq g(0)e^{b\lt(1 - \frac{1}{(t+1)} \rt)} \leq g(0) e^{b}.
\]
Thus we obtain
\[
f(t) \leq f(0)e^b (t+1)^a.
\]

\noindent (ii) Case $\beta < 1$: Dividing the inequality \eqref{est_diff} by $f^\beta$, we get
\[
\frac{1}{1 - \beta}\lt(f^{1-\beta}\rt)^\prime \leq \frac{a}{(t+1)}f^{1-\beta} + \frac{b}{(t+1)^{2\beta}}.
\]
Set $g := f^{1-\beta}$, then we find that $g$ satisfies
\[
\lt(\frac{g(t)}{(t+1)^{a(1-\beta)}} \rt)^{\prime} \leq \frac{b(1-\beta)}{(t+1)^{2\beta + a(1-\beta)}} \quad \mbox{with} \quad g(0) = f(0)^{1-\beta},
\]
due to $1-\beta > 0$. Integrating the above inequality over $[0,t]$, we obtain
\bq\label{est_diff2}
\frac{g(t)}{(t+1)^{a(1-\beta)}} \leq g(0) + b(1-\beta)\int_0^t \frac{1}{(s+1)^{2\beta + a(1-\beta)}}\,ds.
\eq
If $2\beta + a(1-\beta) = 1$, it follows from \eqref{est_diff2} that
\[
\frac{g(t)}{(t+1)^{1-2\beta}} \leq g(0) + (1-2\beta)\int_0^t \frac{1}{s+1}\,ds = g(0) + (1-2\beta)\ln(t+1),
\]
that is, $g$ satisfies
\[
g(t) \leq g(0)(t+1)^{1-2\beta} + (1-2\beta)(t+1)^{1-2\beta}\ln(t+1).
\]
On the other hand, if $2\beta + a(1-\beta) \neq 1$, we find
\[
g(t) \leq g(0)(t+1)^{a(1-\beta)} + \frac{b(1-\beta)}{1 - (2\beta + a(1-\beta))}\lt( (t+1)^{1-2\beta} - (t+1)^{a(1-\beta)} \rt).
\]
We finally substitute $g$ with $f^{1-\beta}$ to complete the proof.
\end{proof}
%%%%%%%%%%%%%%%%%%%%%%%%%%%%%%%%%%%%%%%%%%%%%%%%%%%%%%%%%%%%%%%%%%%%%%%%%%%%%%%%%%%%%%%%%%%%%%%%%%%%%%%%%%%%%%%%%%%%%%%%%%%%%%%%%%%%%%%%%%%%%%
%
%
%    Acknowledgments
%
%
%%%%%%%%%%%%%%%%%%%%%%%%%%%%%%%%%%%%%%%%%%%%%%%%%%%%%%%%%%%%%%%%%%%%%%%%%%%%%%%%%%%%%%%%%%%%%%%%%%%%%%%%%%%%%%%%%%%%%%%%%%%%%%%%%%%%%%%%%%%%%%%%

\section*{Acknowledgments}
The author also thanks Professor Jos\'e A. Carrillo and Professor Laurent Desvillettes for helpful discussion and valuable comments. The author was supported by Engineering and Physical Sciences Research Council(EP/K00804/1). This work is also supported by the Alexander Humboldt Foundation through the Humboldt Research Fellowship for Postdoctoral Researchers. \newline

%%%%%%%%%%%%%%%%%%%%%%%%%%%%%%%%%%%%%%%%%%%%%%%%%%%%%%%%%%%%%%%%%%%%%%%%%%%%%%%%%%%%%%%%%%%%%%%%%%%%%%%%%%%%%%%%%%%%%%%%%%%%%%%%%%%%%%%%%%%%%%
%
%
%         Section: References
%
%
%%%%%%%%%%%%%%%%%%%%%%%%%%%%%%%%%%%%%%%%%%%%%%%%%%%%%%%%%%%%%%%%%%%%%%%%%%%%%%%%%%%%%%%%%%%%%%%%%%%%%%%%%%%%%%%%%%%%%%%%%%%%%%%%%%%%%%%%%%%%%%%%

\end{document}